\documentclass[11pt]{amsart}
\usepackage[dvipdfmx]{graphicx}
\usepackage{amscd}
\usepackage{color}
\author{Tsukasa Isoshima}
\address{Department of Mathematics, Tokyo Institute of Technology, 2-12-1 Ookayama, Meguro-ku, Tokyo, 152-8551, Japan}
\email{isoshima.t.aa@m.titech.ac.jp}

\author{Masaki Ogawa}
\address{Mathematical science center for co-creative society, Tohoku University, Aoba-6-3 Aramaki, Aoba Ward, Sendai, Miyagi 980-0845}
\email{masaki.ogawa.b7@tohoku.ac.jp}

\subjclass[2020]{57K40, 57R65.}

\usepackage[hang,small,bf]{caption}
\usepackage[subrefformat=parens]{subcaption}
\theoremstyle{plain}
\newtheorem{theorem}{Theorem}[section]
\newtheorem{lemma}[theorem]{Lemma}

\newtheorem{question}[theorem]{Question}

\theoremstyle{definition}
\newtheorem{definition}[theorem]{Definition}

\newtheorem{example}[theorem]{Example}

\theoremstyle{definition}

\begin{document}

%\title{Decomposition of lens spaces with three handlebodies}
\title{Trisections induced by the gluck surgery along certain spun knots}

\date{}
\begin{abstract}
Gay and Meier asked whether or not a trisection diagram obtained by the Gluck twist on a spun or a twist spun 2-knot obtained from some method is standard. In this paper, we depict the trisection diagrams explicitly when the 2-knot is the spun $(2n + 1, 2)$-torus knot, where $n\geq1$, and show that the trisection diagram is standard when $n = 1$. Moreover, we introduce a notion of homologically standard for trisection diagrams and show that the trisection diagram is homologically standard for all $n$.
\end{abstract}

\maketitle

\section{introduction}

A trisection is a decomposition of a 4-manifold introduced by Gay and Kirby \cite{GK}. This decomposition is used to study surface knots and 4-manifolds. One of the fields of studying trisections is the classification of trisections.

Similarly, a Heegaard splitting is a decomposition of a 3-manifold, and the classification of Heegaard splittings is a longstanding problem. A well-known result in this area is due to Waldhausen \cite{W}, which states the uniqueness of Heegaard splittings of the 3-sphere for each genus. Bonahon and Otal showed a similar theorem for lens spaces \cite{BO}.

Following the results of a Heegaard splitting, there are problems in the classification of trisections. By spinning non-isotopic Heegaard splittings of a Seifert fibered space, one can obtain mutually non-isotopic trisections of the spin of Seifert fibered spaces \cite{I}. Islambouli showed that each of these trisections is isotopic after one stabilization, based on the stable equivalence of trisections.

In \cite{MSA}, Meier, Schirmer, and Zupan conjectured that each trisection of the standard 4-sphere is unique, which is a generalization of Waldhausen's theorem about Heegaard splittings. One of the cases of this conjecture is a trisection of the 4-sphere obtained by the Gluck surgery.
%In this paper, we consider some cases.

The Gluck surgery is an operation performed along a 2-knot in a 4-manifold. For a 2-knot in $S^4$, a 4-manifold obtained by the Gluck surgery is a homotopy 4-sphere. It is known that the 4-manifold obtained by the Gluck surgery along a spun 2-knot is also the 4-sphere. Gay and Meier described a way to construct a trisection diagram of the 4-manifold and asked if the trisection diagram is standard \cite{GM1}. This question is a special case of the conjecture mentioned above, and the trisection diagram is considered as a potential counterexample to the conjecture, even in the case of the spun trefoil knot. However, in this paper, we show the following result:

\begin{theorem}\label{thm1}
	The trisection diagrams $\mathcal{D}_{x} \cup \mathcal{D}_{S(t(3,2))}$ and $\overline{\mathcal{D}_{x}} \cup \mathcal{D}_{S(t(3,2))}$ ($x=a,b,c$) are standard.
\end{theorem}

This theorem answers affirmatively the question given by Gay and Meier in the simplest case. Note that $\mathcal{D}_{x}$ and $\overline{\mathcal{D}_{x}}$ are  described in Figure \ref{fig:D_x,overline{D_x}}.

To prove this theorem, we perform many handle slides on the obtained trisection diagram. Determining whether a diagram is standard or not by handle sliding alone is challenging, so we consider homological information about trisection diagrams. We obtain the following result from this homological information:

\begin{theorem}\label{thm2}
	The trisection diagrams $\mathcal{D}_{x} \cup \mathcal{D}_{S(t(2n+1,2))}$ and $\overline{\mathcal{D}_{x}} \cup \mathcal{D}_{S(t(2n+1,2))}$ ($x=a,b,c$) are homologically standard.
\end{theorem}
This means that the trisection diagrams obtained by the Gluck surgery in Theorem \ref{thm2} may be standard.

This paper is organized as follows:
In Section 2, we review some definitions of trisections and relative trisections. After that, we construct trisection diagrams obtained by the Gluck twist in Section 3. Then we show Theorem \ref{thm1} in Section 4. In Section 5, we introduce a notion of homological standard and show Theorem \ref{thm2}.

\section{preliminaries}
%In this section, we recall the notion of a trisection and a bridge trisection. 
%First, we review the definition of a trisection and a bridge trisection.
In this paper, we suppose that all 4-manifolds are smooth, compact, connected, and orientable, and a surface link in a 4-manifold is smoothly embedded.

\subsection{Trisections and relative trisections}
A trisection of a 4-manifold is a decomposition of the 4-manifold into three 1-handlebodies. This is an analogy of a Heegaard splitting of a 3-manifold.
\begin{definition}
	Let $X$ be a closed 4-manifold.
	A decomposition $X=X_1\cup X_2\cup X_3$ is called a {\it trisection} if the following holds:
	\begin{itemize}
		\item $X_i\cong \natural ^{k_i} S^1\times B^3$ for $i=1, 2, 3$,
		\item $X_i\cap X_j\cong H_g$ for $i\neq j$, and
		\item $X_1\cap X_2\cap X_3\cong \Sigma_g$
	\end{itemize}
	, where $H_g$ is a 3-dimensional genus $g$ handlebody and $\Sigma_g$ is a genus $g$ orientable, closed surface.
	In this case, we call this trisection a $(g; k_1, k_2, k_3)$-trisection. 
\end{definition}

Any closed 4-manifold admits a trisection \cite{GK}. For a 4-manifold with non-empty boundaries, a relative trisection was defined in \cite{CGP}.
Before reviewing the definition of a relative trisection, we give some notations.
We decompose the boundary of the disk $D=\{re^{i\theta} | r\in [0, 1] , -\pi/3\leq \theta \leq \pi / 3 \}$ in $\mathbb{C}$ as follows:
\[
	\partial ^{-} D = \{re \in \partial D | \theta = -\pi/3\}, 
\] 
\[
	\partial^{0} D = \{e^{i\theta}\in \partial D\}, \text{and}
\]
\[
	\partial^{+} D = \{re \in \partial D | \theta = \pi/3\}.
\]
For the genus $p$ surface $P$ with $b$ boundary, $U=P\times D$ is diffeomorphic to $\natural ^{2p+b-1} S^1\times B^3$.
Also, its boundary is decomposed into $\partial ^0 U = (P\times \partial ^{-} D) \cup (\partial P \times D)$ and $\partial ^{\pm} U = P\times \partial^{\pm} D$.
Let $V_n$ be $\natural ^n (S^1\times B^3)$ and  $\partial V_n = H_s^{-}\cup H_{s}^{+}$ a genus $n+s$ Heegaard splitting.

We set $s=g-k+p+b-1$ and $n=k-2p-b+1$. Then $Z_k\cong U\natural V_n\cong \natural ^k S^1\times B^3$.
We note that 
\[
	\partial Z_k = Y^{+}_{g, k ; p, b}\cup Y^{0}_{g, k ; p, b}\cup Y^{-}_{g, k ; p, b}
\]
, where $Y^{\pm}_{g, k ; p, b}=\partial ^{\pm} U\natural H_{s}^{\pm}$ and $Y^{0}_{g, k ; p, b}= \partial ^0 U$.
After the above preparation, we define a relative trisection as follows.
\begin{definition}
	Let $W$ be a 4-manifold with connected boundary. We call a decomposition $W=W_1\cup W_2\cup W_3$ a $(g, k; p, b)$-{\it relative trisection} of W if: 
	\begin{itemize}
		\item $W_i\cong Z_k$ for $i=1, 2, 3$, 
		\item $W_i\cap W_j\cong Y^{+}_{g, k ; p, b}$ and $W_i\cap W_{i-1}\cong Y^{-}_{g, k ; p, b}$ for $i=1, 2, 3$.
	\end{itemize}
	As a consequence of the above condition, $W_1\cap W_2\cap W_3$ is a genus $g$ surface with $b$ boundary components. 
\end{definition}

\subsection{Doubly-pointed Heegaard diagrams and trisection diagrams}

%A classical knot in a 3-manifold can be described as a doubly pointed Heegaard diagram. This is induced by a bridge decomposition with respect to a Heegaard splitting of a 3-manifold.
Let $K$ be a knot in a 3-manifold $M$, $\Sigma$ a Heegaard surface, and $(\Sigma; \alpha, \beta)$ a genus $g$ Heegaard diagram of $M$.
Suppose that $K$ is in a bridge position with respect to $\Sigma$.
We denote $\bold{x}=K\cap \Sigma$.
Then we call $(\Sigma; \alpha, \beta, \bold{x})$ a {\it pointed Heegaard diagram}.
Particularly, it is called a {\it doubly pointed Heegaard diagram} if $\bold{x}$ is two points.

For a surface knot in a 4-manifold, Meier and Zupan showed that every closed surface embedded in a 4-manifold can be put into a bridge trisected position \cite{MZ1}.
They also showed a 2-knot in a 4-manifold can be put into a 1-bridge position (Theorem 2 in \cite{MZ1}).
More precisely, see \cite{MZ1}.
%We review the notion of a bridge trisection.

%\begin{definition}
%	Let $X_1\cup X_2\cup X_3$ be a trisection of a 4-manifold $X$ and $\mathcal{K}$ a 2-knot in $X$.
%	The decomposition $(X, \mathcal{K})=(X_1, D_1)\cup (X_2, D_2)\cup (X_3, D_3)$ is called {\it1-bridge trisection} of $\mathcal{K}$ with respect to $X_1\cup X_2\cup X_3$ if the following holds:
%	\begin{itemize}
%		\item $(X_i, D_i)$ is a trivial disk (i.e. each of $D_i$ are isotopic into the boundary), 
%		\item $H_{ij}\cap \mathcal{K}$ is a trivial arc, where $H_{ij}=X_i\cap X_j$.
%	\end{itemize}
%\end{definition}

We then recall a doubly pointed trisection diagram introduced in \cite{GM1}.
A trisection diagram is a 4-tuple $(\Sigma; \alpha, \beta, \gamma)$ such that each of $(\Sigma; \alpha, \beta)$, $(\Sigma;  \beta, \gamma)$ and $(\Sigma; \alpha, \gamma)$ is a Heegaard diagram of $\#^{k_i} S^1\times S^2$ for $i=1, 2, 3$ respectively.
\begin{definition}
We call $(\Sigma; \alpha, \beta, \gamma, \bold{x})$ a {\it doubly pointed trisection diagram} if the following holds:
\begin{itemize}
	\item $(\Sigma; \alpha, \beta, \gamma)$ is a trisection diagram.
	\item $\bold{x}$ is a set of disjoint two points on $\Sigma$ disjoint from $\alpha$, $\beta$ and $\gamma$.
\end{itemize}
\end{definition}

It is known that a doubly pointed trisection diagram determines a 2-knot uniquely (Proposition 4.5 in \cite{GM1}).

Recall that a $(g, k; p, b)$ {\it relative trisection diagram} is a 4-tuple $(\Sigma; \alpha, \beta, \gamma)$ consists of a genus $g$ surface with $b$ boundary components and three sets of $g-p$ simple closed curves such that each triple $(\Sigma; \alpha, \beta)$, $(\Sigma; \alpha, \gamma)$ and $(\Sigma; \beta, \gamma)$ are slide equivalent to the following figure (Figure \ref{reltridiag}). We then obtain a diagram of a relative trisection that has information on the open book decomposition of its boundary.
That is called an {\it arced relative trisection diagram}. More precisely, see Definition 2.12 in \cite{GM1}.

	  \begin{figure}[h]
		\centering
			\includegraphics[scale=0.6]{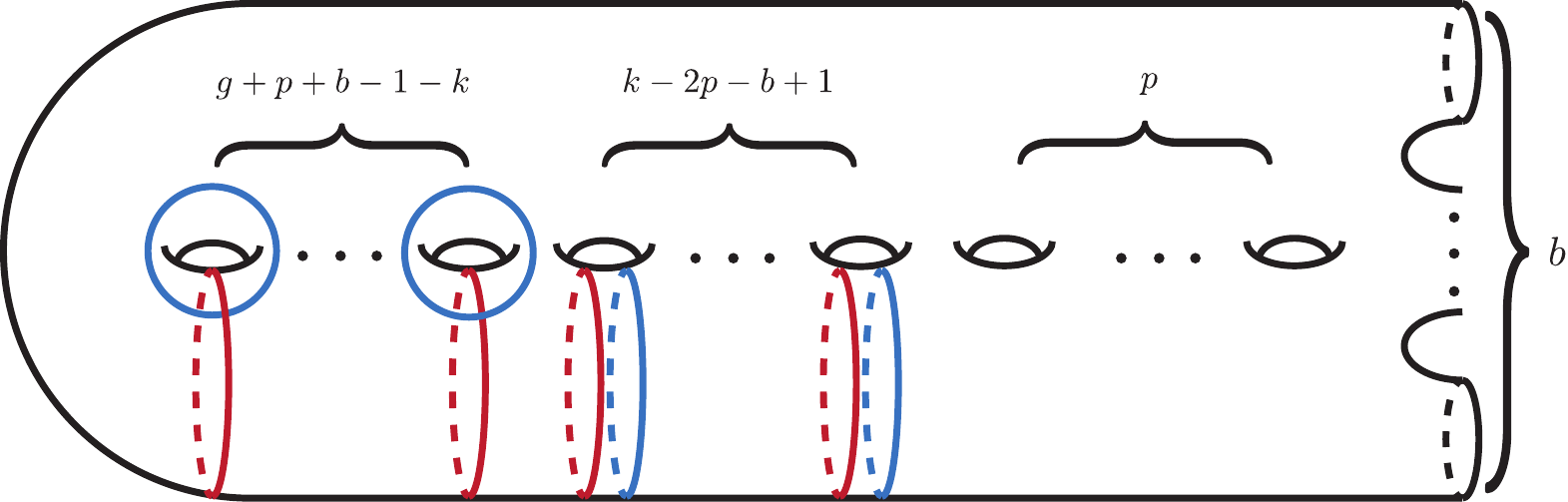}
			\caption
			{After handlesliding each of $(\Sigma; \alpha, \beta)$, $(\Sigma; \alpha, \gamma)$ and $(\Sigma; \beta, \gamma)$, we obtain this diagram. The color represents a different set of curves.}
			\label{reltridiag}
	\end{figure}

\section{Trisections obtained by the Gluck surgery}
The Gluck surgery is a regluing operation on a 2-knot in a 4-manifold.
%This is a generalization of the Dehn surgery.
Let $\mathcal{K}$ be a 2-knot in a 4-manifold $X$ with normal Euler number $0$ and $N(\mathcal{K})$ a regular neighborhood of $\mathcal{K}$.
%The boundary of the exterior of $\mathcal{K}$ is diffeomorphic to $S^2\times S^1$.
Let $f$ be the non-trivial self diffeomorphism of $S^2\times S^1$, namely $f(z, e^{i\theta})=(z,ze^{i\theta})$.
Then we can reglue $S^2\times D^2$ to $X\setminus int N(\mathcal{K})$ by the diffeomorphism $f$.
This operation is called the {\it Gluck surgery}. 
We denote the 4-manifold obtained by Gluck twisting  $\Sigma_{\mathcal{K}}(X)$.
%%%%%%%%%%%%%%%%%%%まだ書いてないもの%%%%%%%%%%%%%%%%%%%
%・Meierのスパン
%・Gay-Meierの手法
%・arcのアルゴリズム
%・dpHdの証明
%%%%%%%%%%%%%%%%%%%まだ書いてないもの%%%%%%%%%%%%%%%%%%%

At the beginning, we define a term that is used in the main theorem and review a question that is related to the main theorem.

\begin{definition}
Let $X$ be a closed 4-manifold which is diffeomorphic to $S^4$. A trisection diagram of $X$ is \textit{standard} if it is the stabilization of the genus 0 trisection diagram up to surface diffeomorphism and handle slide.
\end{definition}

\begin{question}[Question 6.4 in \cite{GM1}]\label{que:Gay-Meier}
Is the trisection diagram for Gluck twisting on the spin or twist spin of a non-trivial classical knot in $S^3$ obtained by the methods of \cite{M} and \cite{GM1} \textbf{ever} standard?
\end{question}

In this section, we construct the trisection diagram in Question \ref{que:Gay-Meier} for the spun $(2n+1,2)$-torus knot $S(t(2n+1,2))$, where $n \geq 1$. A recipe for constructing the trisection diagram for $S(t(2n+1,2))$ is as follows.

\begin{enumerate}
\item Equip a doubly pointed Heegaard diagram of ($S^3,t(2n+1,2)$).
\item Construct a doubly pointed trisection diagram of ($S^4,S(t(2n+1,2))$) from the doubly pointed Heegaard diagram using Meier's method.
\item Construct a relative trisection diagram of $S^4-S(t(2n+1,2))$ from the doubly pointed trisection diagram by removing the two base points.
\item Construct an arced relative trisection diagram of $S^4-S(t(2n+1,2))$ from the relative trisection diagram using an algorithm for arcs.
\item Construct the trisection diagram for Gluck twisting on $S(t(2n+1,2))$ by gluing Figure \ref{fig:Gluck} and the arced relative trisection diagram (\cite{GM1}).
\end{enumerate}

From now on, we recall the three methods in the recipe. Then, we construct the trisection diagram according to the above recipe in order.

\subsection{Trisection diagrams for the spin of 3-manifolds}
In \cite{M}, Meier constructed a trisection diagram of a 4-manifold $S(M)$, the spin of a closed, connected, orientable 3-manifold $M$, using a Heegaard diagram of $M$.

\begin{theorem}[Theorem 1.4 in \cite{M}]\label{spun}%定理番号pub版で確認済み
Let $(\Sigma,\delta,\epsilon)$ be a genus $g$ Heegaard diagram of $M$, where the collection of simple closed curves $\epsilon$ is depicted in the left of Figure \ref{meierspun}. Then, for three collections of simple closed curves $\alpha$, $\beta$ and $\gamma$ depicted in the right of Figure \ref{meierspun}, the 4-tuple $(\Sigma,\alpha,\beta,\gamma)$ is a genus $3g$ trisection diagram of $S(M)$.
\end{theorem}

	  \begin{figure}[h]
		\centering
			\includegraphics[scale=0.6]{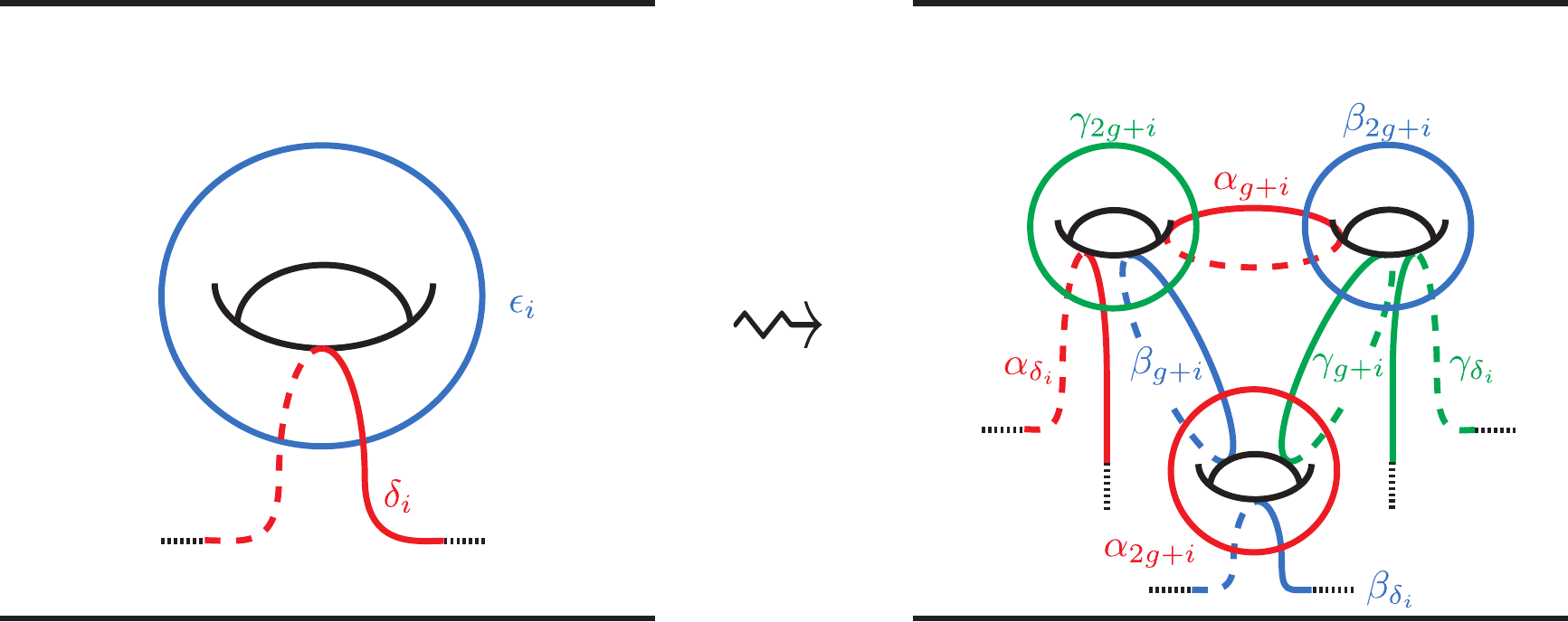}
			\caption
			{We can construct a trisection diagram of a trisection obtained by the spin of a Heegaard splitting.}
			\label{meierspun}
	\end{figure}

%\begin{figure}[h]
%\begin{tabular}{cc}
%\begin{minipage}{0.4\hsize}%小さくすると寄る
%\begin{center}
%\includegraphics[width=8cm, height=3cm, keepaspectratio, scale=1]{meier_spun1.pdf}
%\end{center}
%\setlength{\captionmargin}{50pt}
%\subcaption{}
%\label{fig:meier_spun1}
%\end{minipage} 
%\begin{minipage}{0.4\hsize}
%\begin{center}
%\includegraphics[width=8cm, height=4cm, keepaspectratio, scale=1]{meier_spun2.pdf}
%\end{center}
%\setlength{\captionmargin}{50pt}
%\subcaption{}
%\label{fig:meier_spun2}
%\end{minipage}
%\end{tabular}
%\setlength{\captionmargin}{50pt}
%\caption{}
%\label{fig:bbb}
%\end{figure}%どうやれば小文字になる？どうやれば中央に来る？

He also constructed a doubly pointed trisection diagram of $(S(M),S(K))$ from a doubly pointed Heegaard diagram of $(M,K)$ in the same way, where $K$ is a classical knot in $M$ and $S(K)$ is the spin of $K$ (Theorem 1.5 in \cite{M}).%定理番号pub版で確認済み

%\begin{thm}[Theorem \cite{MR3917737}]
%Let $(\Sigma,\delta,\epsilon,z,w)$ be a genus $g$ doubly pointed Heegaard diagram of a 3-manifold $M$ and a classical knot $K$ depicted in Figure \ref{}. Then, a 6-tuple $(\Sigma,\alpha,\beta,\gamma,z,w)$ depicted in Figure \ref{} is a genus $3g$ doubly pointed trisection diagram of $(S(M),S(K))$.
%\end{thm}

\subsection{Describing trisection diagrams for Gluck twisting}
In \cite{GM1}, Gay and Meier described a trisection diagram of $\Sigma_K(S^4)$ using an arced relative trisection diagram of $S^4-K$.

\begin{theorem}[Lemma 5.5 in \cite{GM1}]\label{1-bridge}
Let $K$ be a 2-knot in $S^4$ which is in 1-bridge position. Then, a trisection diagram of $\Sigma_K(S^4)$ is obtained by gluing Figure \ref{fig:Gluck} and an arced relative trisection diagram of $S^4-K$.
\end{theorem}

Note that the trisection diagram of $\Sigma_K(S^4)$ is independent of the choice of arcs of a relative trisection diagram of $S^4-K$. 

\subsection{An algorithm of drawing arcs for relative trisection diagrams} 
An algorithm for drawing arcs of a relative trisection diagram is developed in \cite{CGP2}.

\begin{theorem}[Theorem 5 in \cite{CGP2}]\label{algorithm}
Let $(\Sigma,\alpha,\beta,\gamma)$ be a relative trisection diagram. Also let $\Sigma_\alpha$ be the surface obtained by surgering $\Sigma$ along $\alpha$ and $\phi \colon \Sigma-\alpha \to \Sigma_\alpha$ an associated embedding. Then, arcs $a$, $b$ and $c$ in $\Sigma$ are obtained by following the procedure below in order.%cut surfaceは商で定義されるから，埋め込みは逆．
\begin{enumerate}
\item There exists a collection of properly embedded arcs $\delta$ in $\Sigma_\alpha$ such that cutting $\Sigma_\alpha$ along $\delta$ produces the disk. Then, let $a$ denote a collection of properly embedded arcs in $\Sigma-\alpha$ such that $\delta$ is isotopic to $\phi(a)$ in $\Sigma_\alpha$.
\item A collection of arcs in $\Sigma$ which does not intersect $\beta$ geometrically is obtained by sliding a copy of $a$ over $\alpha$. Then, let $b$ denote the collection of arcs. Note that in this operation, sliding a curve of $\beta$ over other $\beta$ curves can be performed if necessary.
\item A collection of arcs in $\Sigma$ which does not intersect $\gamma$ geometrically is obtained by sliding a copy of $b$ over $\beta$. Then, let $c$ denote the collection of arcs. Note that in this operation, sliding a curve of $\gamma$ over other $\gamma$ curves can be performed if necessary.
\end{enumerate}
\end{theorem}

\subsection{Constructing trisection diagrams for Gluck twisting}

In this subsection, we will finally construct the trisection diagram according to the recipe described above, utilizing the methods explained in the three previous subsections. Here, we will use the notation $t_c$ to represent the right-handed Dehn twist along a simple closed curve $c$.

%\begin{lemma}\label{lem:(p,-2)_dpHd}
%Let $(\Sigma, \alpha, \beta, z, w)$ be a doubly-pointed Heegaard diagram of $(S^3, t(3,-2))$ depicted in Figure \ref{fig:(p,-2)_dpHd}. Then, for $n \ge 1$, $(\Sigma, \alpha, t_{\delta}^{n-1}(\beta), z, w)$ is a doubly-pointed Heegaard diagram of $(S^3, t(2n+1,-2))$, where $\delta$ is the curve depicted in Figure \ref{fig:(p,-2)_dpHd}.
%\end{lemma}

\begin{lemma}\label{lem:(p,-2)_dpHd}
Let $(\Sigma, \epsilon, \zeta, z, w)$ be a doubly-pointed Heegaard diagram of $(S^3, t(3,2))$ depicted in Figure \ref{fig:(p,-2)_dpHd}. Then, for $n \ge 1$, $(\Sigma, \epsilon, t_{\delta}^{-(n-1)}(\zeta), z, w)$ is a doubly-pointed Heegaard diagram of $(S^3, t(2n+1,2))$, where $\delta$ is the curve depicted in Figure \ref{fig:(p,-2)_dpHd}.
\end{lemma}

\begin{proof}
We can check, using a method in \cite{O}, that the 4-tuple $(\Sigma, \epsilon, t_{\delta}^{-1}(\zeta), z, w)$ is a doubly-pointed Heegaard diagram of $(S^3, t(5,2))$. If we perform this operation $n-1$ times, the proof will be completed.
\end{proof}

\begin{figure}[h]
\begin{center}
\includegraphics[scale=0.7]{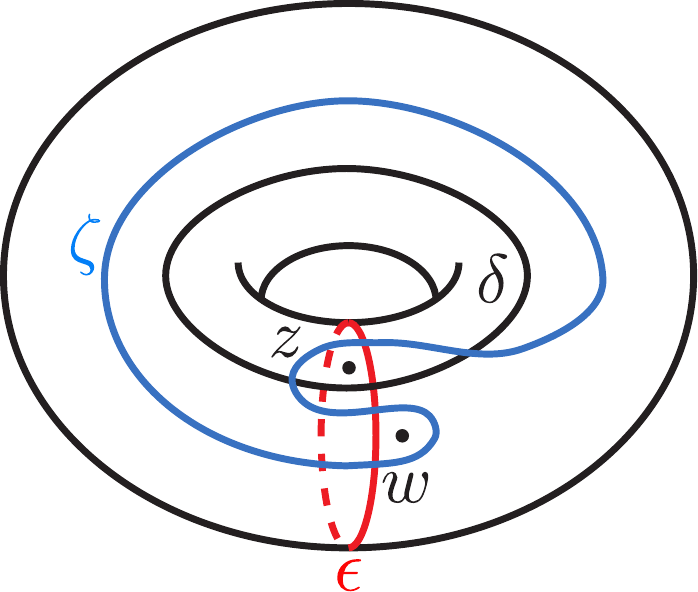}
\end{center}
\caption{A doubly-pointed Heegaard diagram of $(S^3, t(3,2))$ and the reference curve $\delta$.}
\label{fig:(p,-2)_dpHd}
\end{figure}

The following lemma is obtained from Lemma \ref{lem:(p,-2)_dpHd} by the method in \cite{M}. More precisely, we construct a doubly pointed trisection diagram from the doubly pointed Heegaard diagram as prescribed in Lemma \ref{lem:(p,-2)_dpHd}, following the procedure shown in Figure \ref{meierspun}. However, in this case, we need to take into account the presence of the base points. %Note that the doubly-pointed trisection diagram of $(S^4, S(t(3,2)))$ depicted in Figure \ref{fig:S(p,-2)_dptd} can correspond the spun trefoil part of Figure 17 in \cite{KM} by taking the mirror image.
%Note that, we take the mirror image here.
%%%%%bibtexを使っています

\begin{lemma}\label{lem:(p,-2)_dptd}
Let $(\Sigma, \alpha, \beta, \gamma, z, w)$ be a doubly-pointed trisection diagram of $(S^4, S(t(3,2)))$ depicted in Figure \ref{fig:S(p,-2)_dptd}. Then, for $n \ge 1$, $(\Sigma, \alpha', \beta', \gamma', z, w)$ is a doubly-pointed trisection diagram of $(S^4, S(t(2n+1,2)))$, where 
\begin{itemize}
\item $\alpha' = (t_{\delta_1}^{-(n-1)}(\alpha_1), \alpha_2, \alpha_3)$,
\item $\beta' = (t_{\delta_2}^{-(n-1)}(\beta_1), \beta_2, \beta_3)$,
\item $\gamma' = (t_{\delta_3}^{-(n-1)}(\gamma_1), \gamma_2, \gamma_3)$,
\end{itemize}
and curves $\delta_1$, $\delta_2$ and $\delta_3$ are them depicted in Figure \ref{fig:S(p,-2)_dptd}.
\end{lemma}

\begin{figure}[h]
\begin{center}
\includegraphics[scale=0.6]{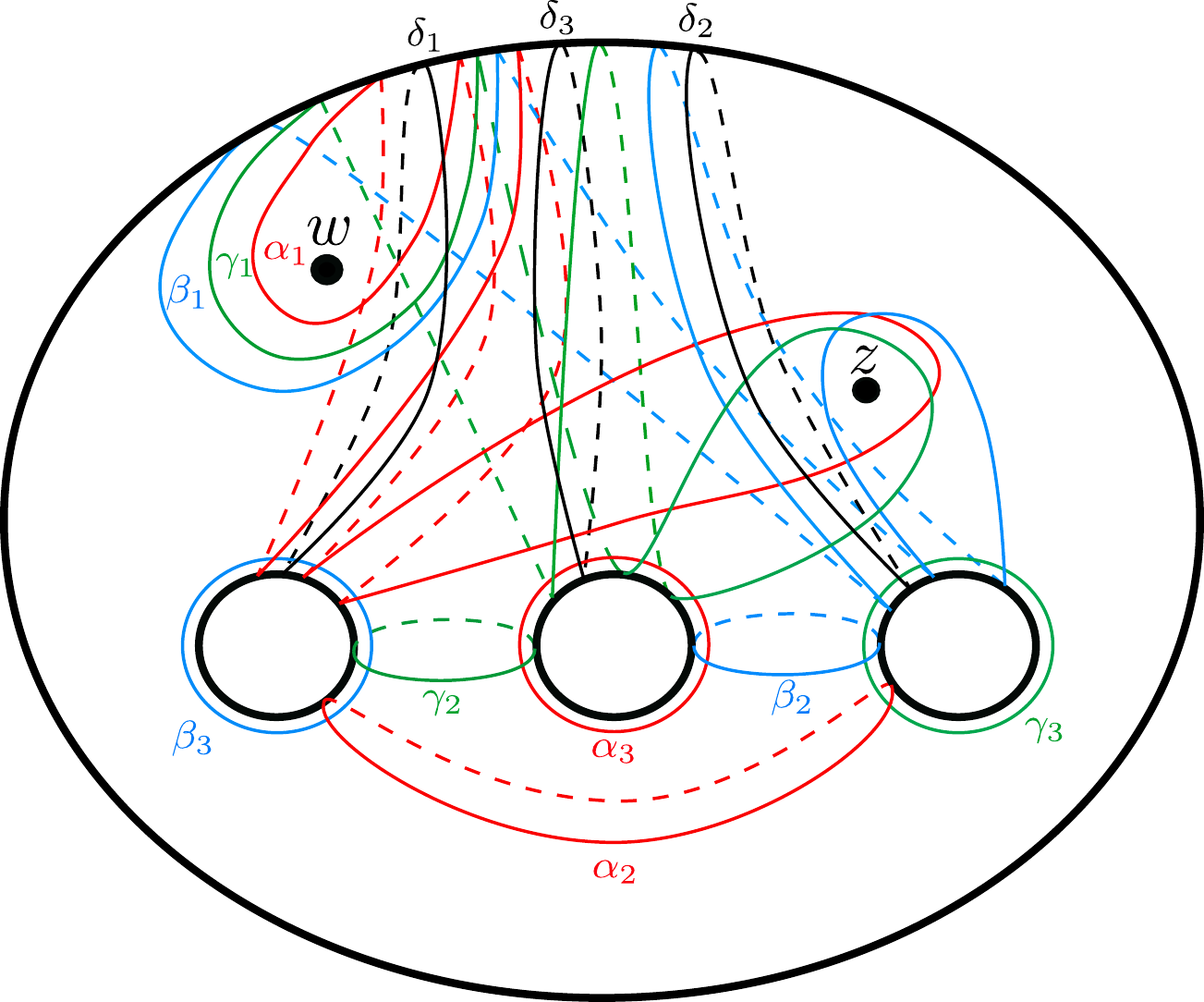}
\end{center}
\setlength{\captionmargin}{50pt}
\caption{A doubly-pointed trisection diagram of $(S^4, S(t(3,2)))$ and the reference curves $\delta_1$, $\delta_2$ and $\delta_3$.}
\label{fig:S(p,-2)_dptd}
\end{figure}

If a doubly pointed trisection diagram is given, a relative trisection diagram of the complement of the 2-knot represented by the diagram can be obtained by excluding a neighborhood of the base points, following the procedure described in subsection 4.1 of \cite{GM1}.  Referring to subsection 4.1 of \cite{GM1} will provide more detailed information on this construction.

Hence, the following lemma is obtained from Lemma \ref{lem:(p,-2)_dptd} by removing the open tubular neighborhood of the two base points $z$ and $w$ in Figure \ref{fig:S(p,-2)_dptd}. Note that since $\delta_1$ can be parallel to $\delta_2$ by sliding over $\alpha_2$ and $\alpha_3$, we can replace the role of $\delta_1$ with the one of $\delta_2$ for $\alpha^{'}$.

\begin{lemma}\label{lem:(p,-2)_rtd}
Let $(\Sigma, \alpha, \beta, \gamma)$ be a relative trisection diagram of $S^4-S(t(3,2))$ depicted in Figure \ref{fig:S(p,-2)_rtd}. Then, for $n \ge 1$, $(\Sigma, \alpha', \beta', \gamma')$ is a relative trisection diagram of $S^4-S(t(2n+1,2))$, where 
\begin{itemize}
\item $\alpha' = (t_{\delta_2}^{-(n-1)}(\alpha_1), \alpha_2, \alpha_3)$,
\item $\beta' = (t_{\delta_2}^{-(n-1)}(\beta_1), \beta_2, \beta_3)$,
\item $\gamma' = (t_{\delta_3}^{-(n-1)}(\gamma_1), \gamma_2, \gamma_3)$,
\end{itemize}
and curves $\delta_2$ and $\delta_3$ are them depicted in Figure \ref{fig:S(p,-2)_rtd}.
\end{lemma}

\begin{figure}[h]
\begin{center}
\includegraphics[scale=0.6]{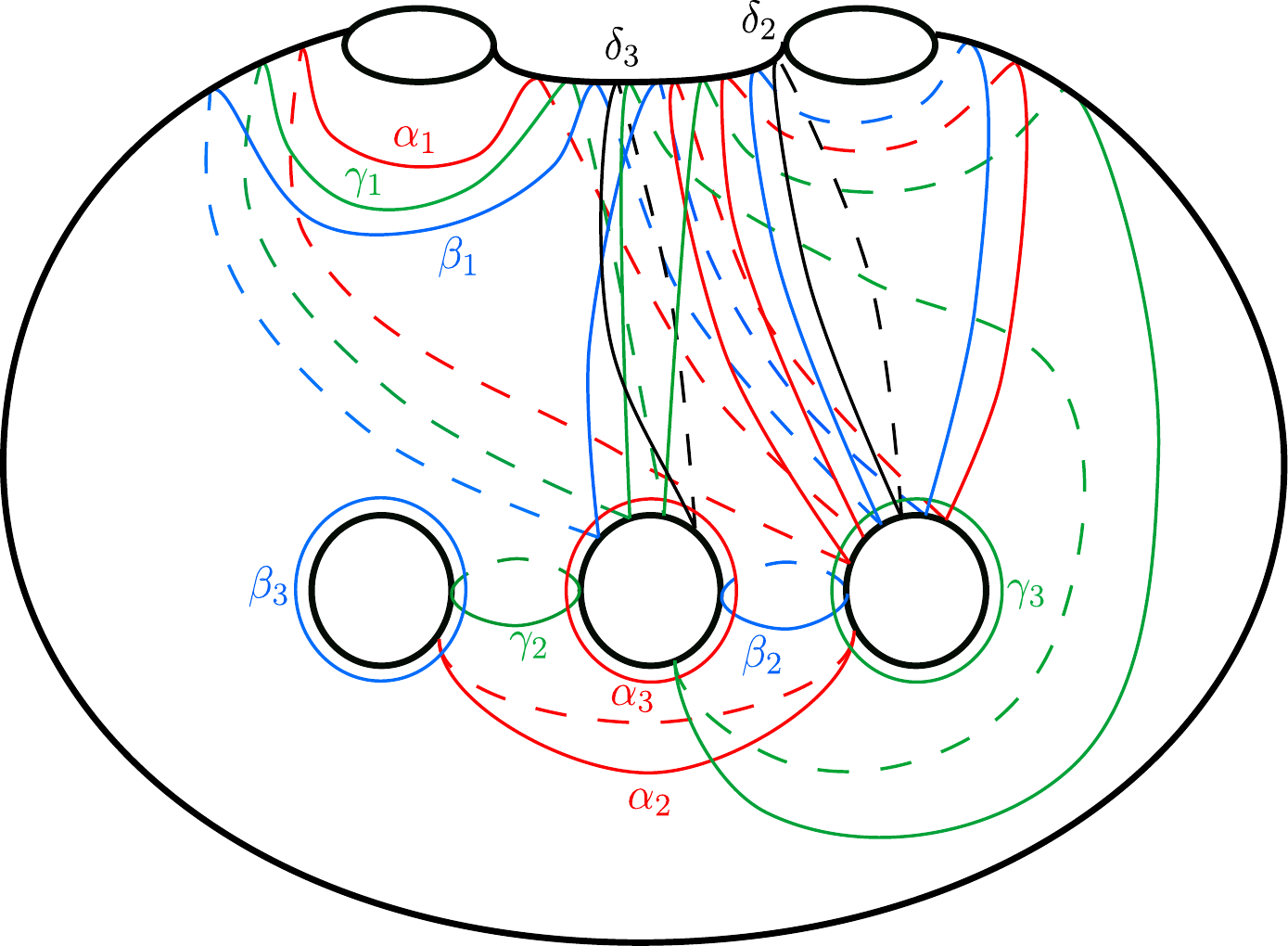}
\end{center}
\setlength{\captionmargin}{50pt}
\caption{A relative trisection diagram of $S^4-S(t(3,2)))$ and the reference curves $\delta_2$ and $\delta_3$.}
\label{fig:S(p,-2)_rtd}
\end{figure}

The following lemma is obtained from Lemma \ref{lem:(p,-2)_rtd} using Theorem \ref{algorithm}.

\begin{lemma}\label{lem:(p,-2)_artd}
Let $(\Sigma, \alpha, \beta, \gamma, a, b, c)$ be an arced relative trisection diagram of $S^4-S(t(3,2))$ depicted in Figure \ref{fig:S(p,-2)_artd}. Then, for $n \ge 1$, $(\Sigma, \alpha', \beta', \gamma', a', b', c')$ is an arced relative trisection diagram of $S^4-S(t(2n+1,2))$, where 
\begin{itemize}
\item $\alpha' = (t_{\delta_2}^{-(n-1)}(\alpha_1), \alpha_2, \alpha_3)$,
\item $\beta' = (t_{\delta_2}^{-(n-1)}(\beta_1), \beta_2, \beta_3)$,
\item $\gamma' = (t_{\delta_3}^{-(n-1)}(\gamma_1), \gamma_2, \gamma_3)$,
\item $a' = t_{\delta_2}^{-(n-1)}(a)$,
\item $b' = t_{\delta_2}^{-(n-1)}(b)$,
\item $c' = t_{\delta_3}^{-(n-1)}(c)$,
\end{itemize}
and curves $\delta_2$ and $\delta_3$ are them depicted in Figure \ref{fig:S(p,-2)_artd}.
\end{lemma}

\begin{proof}
The arcs $a,b,c$ can be drawn by Theorem \ref{algorithm}. In drawing $b$, we can slide $a$ over $\alpha$ curves so that arcs taken in this handle slide do not intersect $\delta_2$. We can draw $c$ similarly. Hence, the arcs $a',b',c'$ can be obtained.
\end{proof}

Let $\mathcal{D}_{S(t(2n+1,2))}$ denote this arced relative trisection diagram.

\begin{figure}[h]
\begin{center}
\includegraphics[width=8cm, height=7cm, keepaspectratio, scale=1]{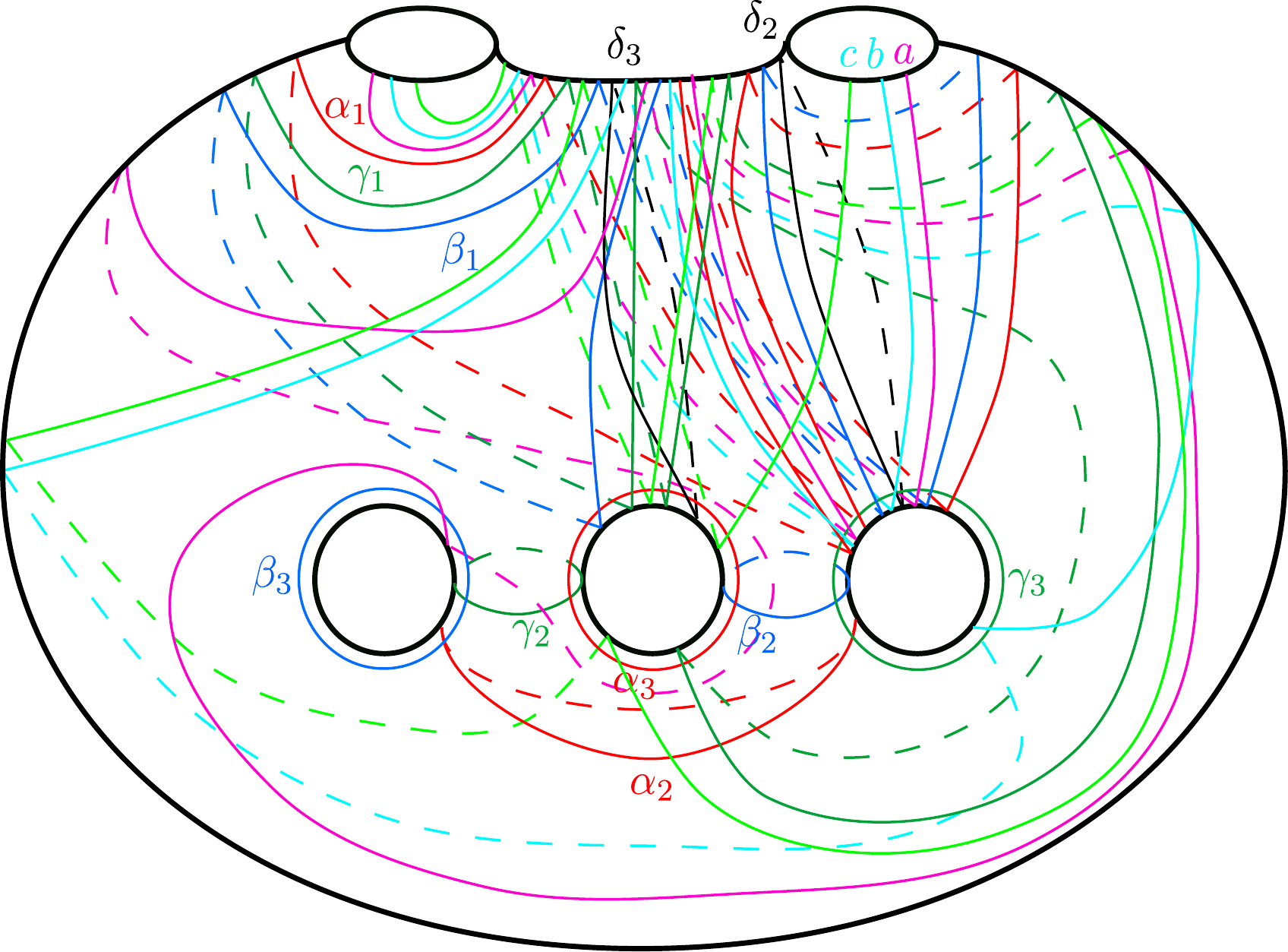}
\end{center}
\setlength{\captionmargin}{50pt}
\caption{An arced relative trisection diagram of $S^4-S(t(3,2)))$ and the reference curves $\delta_2$ and $\delta_3$.}
\label{fig:S(p,-2)_artd}
\end{figure}

\begin{example}
Figure \ref{fig:S(5,-2)_artd} shows the arced relative trisection diagram $\mathcal{D}_{S(t(5,2))}$ obtained from Lemma \ref{lem:(p,-2)_artd} (depicted for each family curve).
\end{example}

\begin{figure}[h]
\centering
\includegraphics[scale=0.22]{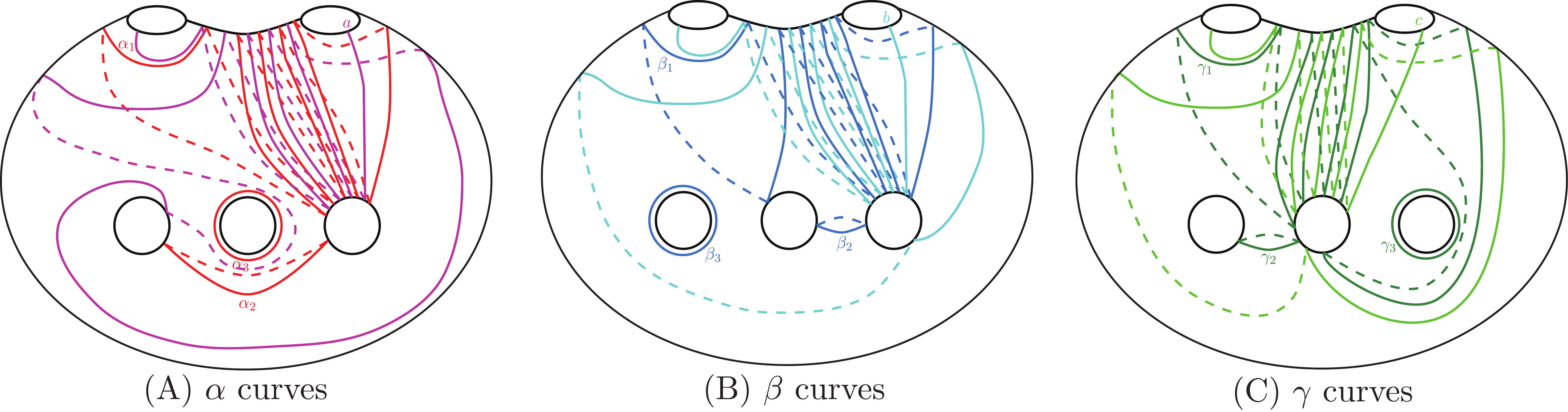}
\setlength{\captionmargin}{50pt}
\caption{The arced relative trisection diagram $\mathcal{D}_{S(t(5, 2))}$}
\label{fig:S(5,-2)_artd}
\end{figure}

The following lemma is obtained by combing Lemma \ref{lem:(p,-2)_artd} and Lemma 5.5 in \cite{GM1}.

\begin{lemma}\label{lem:main}
Let $(\Sigma, \alpha, \beta, \gamma)$ be a trisection diagram obtained by Gluck twisting on the spun trefoil depicted in Figure \ref{fig:main}, using the methods of \cite{GM1} and \cite{M}(Theorem \ref{spun} and \ref{1-bridge}). Then, for $n \ge 1$, $(\Sigma, \alpha', \beta', \gamma')$ is a trisection diagram obtained by Gluck twisting on $S(t(2n+1,2))$ using the same methods, where 
\begin{itemize}
\item $\alpha' = (t_{\delta_2}^{-(n-1)}(\alpha_1), \alpha_2, \alpha_3, t_{\delta_2}^{-(n-1)}(\alpha_4), \alpha_5, \alpha_6)$,
\item $\beta' = (t_{\delta_2}^{-(n-1)}(\beta_1), \beta_2, \beta_3, t_{\delta_2}^{-(n-1)}(\beta_4), \beta_5, \beta_6)$,
\item $\gamma' = (t_{\delta_3}^{-(n-1)}(\gamma_1), \gamma_2, \gamma_3, t_{\delta_3}^{-(n-1)}(\gamma_4), \gamma_5, \gamma_6)$,
\end{itemize}
and curves $\delta_2$ and $\delta_3$ are them depicted in Figure \ref{fig:main}.
\end{lemma}

Let $\mathcal{D} \cup \mathcal{D}_{S(t(2n+1,2))}$ denote the trisection diagram of $\Sigma_{S(t(2n+1,2))}(S^4)$ constructed in Lemma \ref{lem:main} (we write $\mathcal{D}$ for Figure \ref{fig:Gluck}).

\begin{figure}[h]
\begin{center}
\includegraphics[scale=0.6]{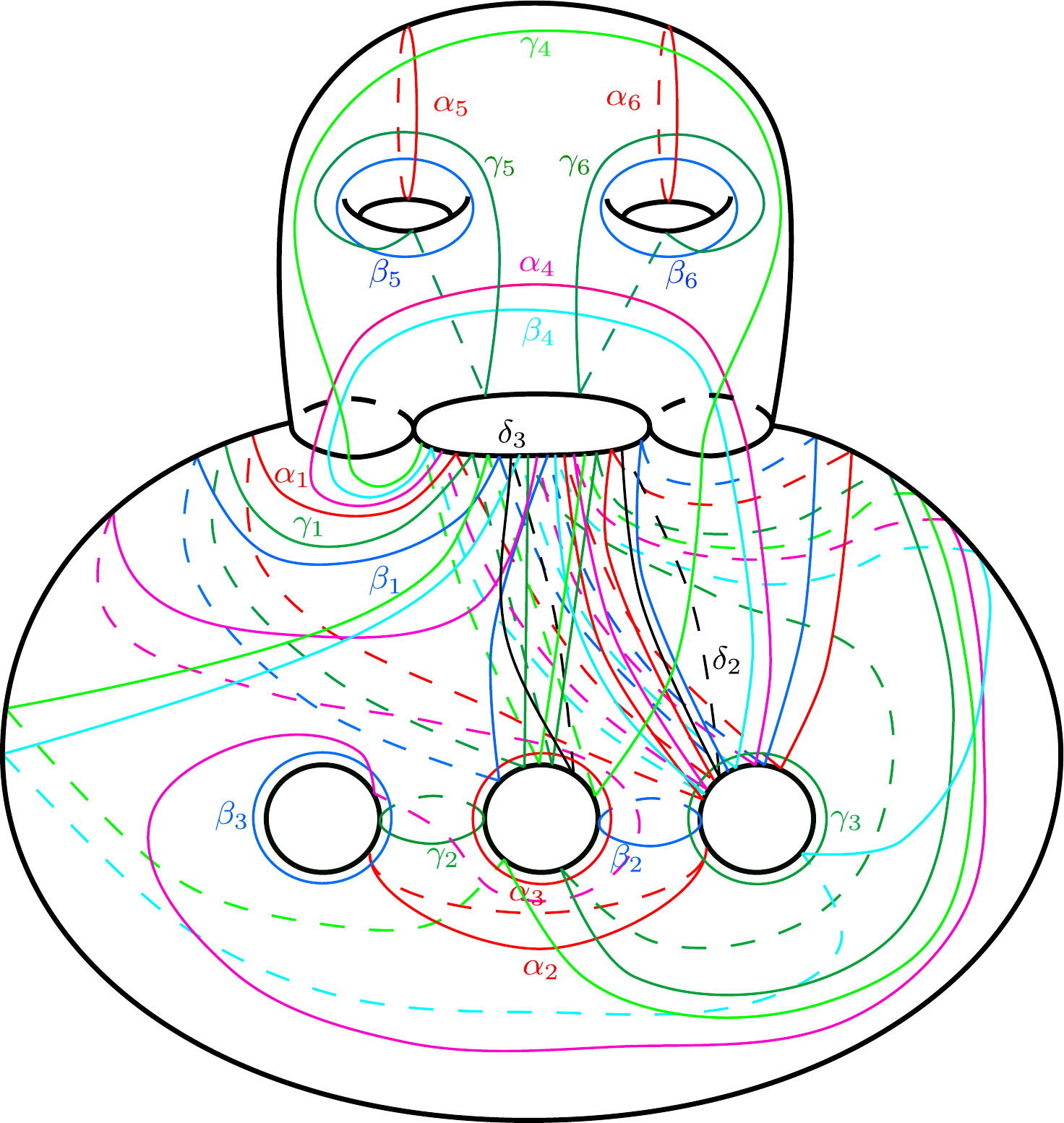}
\end{center}
\setlength{\captionmargin}{50pt}
\caption{A trisection diagram obtained by Gluck twisting on the spun trefoil and the reference curves $\delta_2$ and $\delta_3$.}
\label{fig:main}
\end{figure}

\begin{figure}[h]
\begin{center}
\begin{minipage}{0.3\hsize}
\includegraphics[width=8cm, height=3cm, keepaspectratio, scale=1]{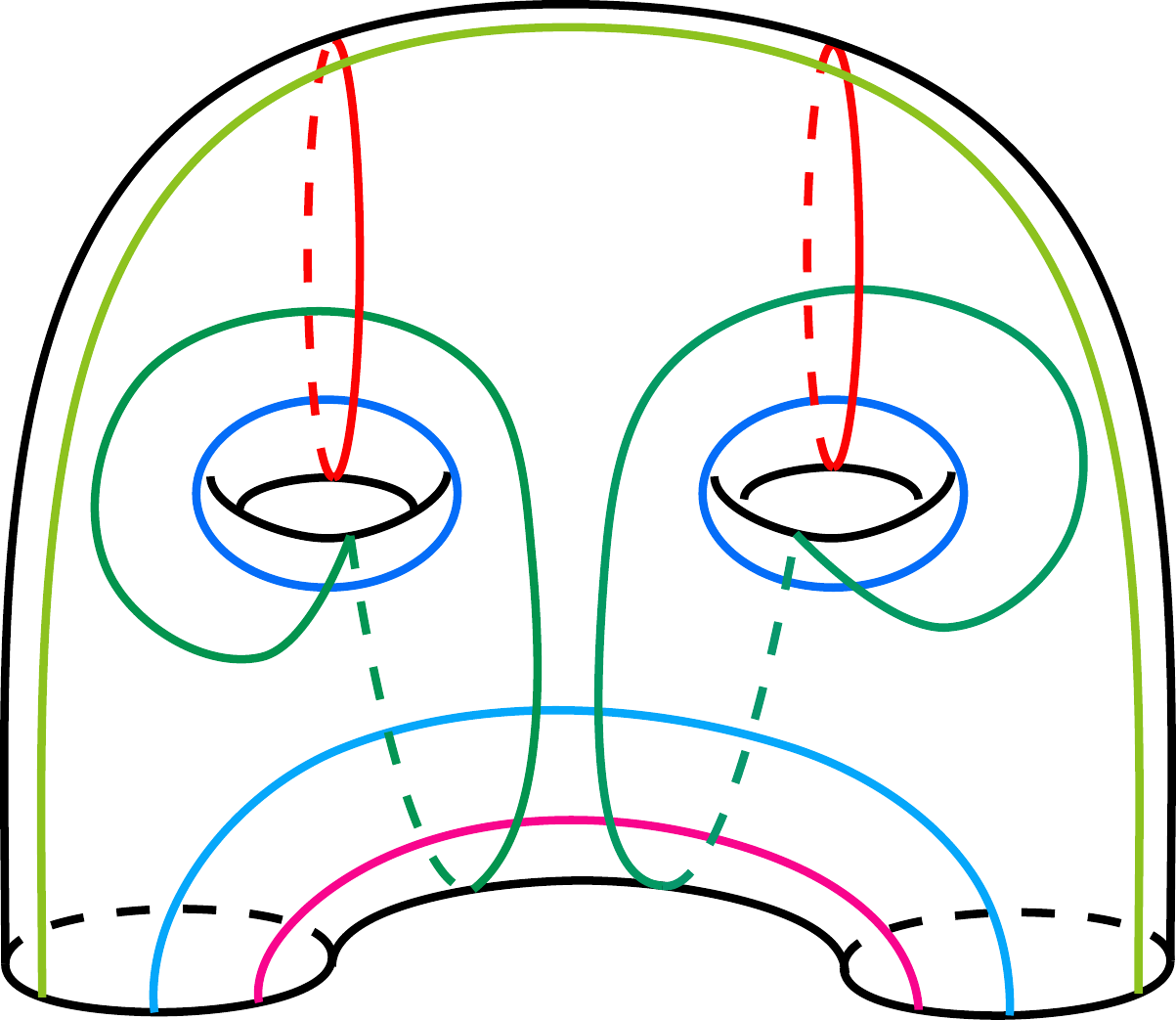}
\setlength{\captionmargin}{50pt}
\subcaption{}
\label{fig:Gluck}
\end{minipage} 
\begin{minipage}{0.5\hsize}
\includegraphics[width=8cm, height=4cm, keepaspectratio, scale=1]{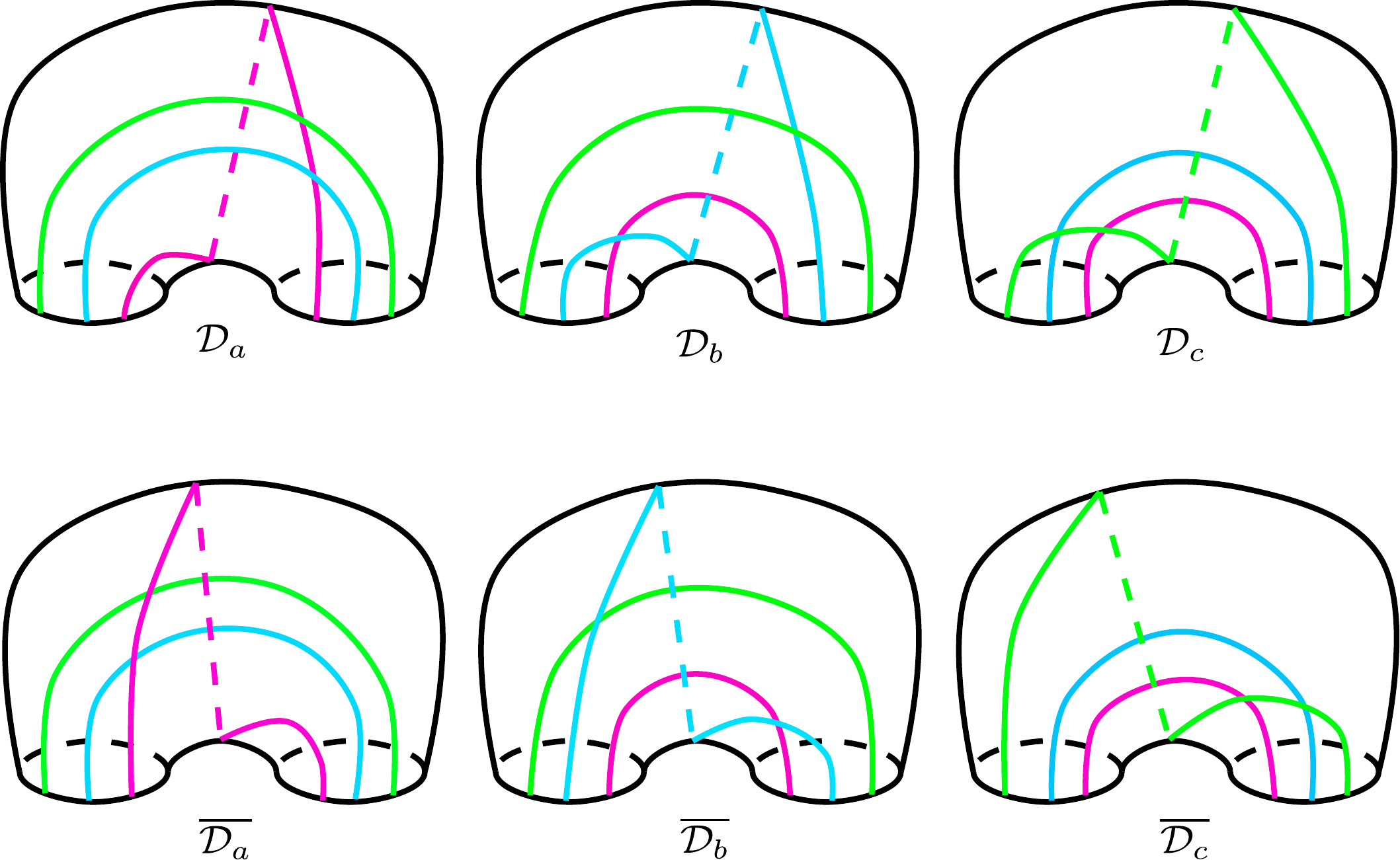}
\setlength{\captionmargin}{50pt}
\subcaption{}
\label{fig:D_x,overline{D_x}}
\end{minipage}
\end{center}
\setlength{\captionmargin}{50pt}
\caption{The figure (a) depicts an arced relative trisection diagram used in performing the Gluck twist, and the figure (b) is obtained by destabilizing the figure (a).}
\label{fig:aaa}
\end{figure}

\section{The case of the spun trefoil}
In this section, we prove Theorem \ref{thm1}. Since destabilizing $\mathcal{D}$ twice provides us with one of Figure \ref{fig:D_x,overline{D_x}}, we consider the case where $\mathcal{D}_{x} \cup \mathcal{D}_{S(t(3,2))}$ and $\overline{\mathcal{D}_{x}} \cup \mathcal{D}_{S(t(3,2))}$.

\begin{proof}[Proof of Theorem \ref{thm1}]
During the proof, we call a name of a simple closed curve obtained by a Dehn twist or a handle slide for a simple closed curve named $c$ also $c$. A trisection diagram is depicted for each family curve as in Figure \ref{fig:start}.

Note that the arced relative trisection diagram $\mathcal{D}_{S(t(3,2))}$ depicted in Figure \ref{fig:S(p,-2)_artd} has $\mathbb{Z}_3$-symmetry since it is 0-annular, namely, the $\gamma$ arcs and $\alpha$ arcs can be paralleled by handle sliding. Thus, it suffices to show that $\mathcal{D}_{b} \cup \mathcal{D}_{S(t(3,2))}$ and $\overline{\mathcal{D}_{b}} \cup \mathcal{D}_{S(t(3,2))}$ are standard. 

To begin with, we show the theorem in the case of $\mathcal{D}_{b}$. Note that in both cases, handle slides over some curves may be performed several times. Also note that in some figures we describe how to take arcs for complicated handle slides by arrows. In Figure \ref{fig:start}, for $\alpha$ curves, slide $\alpha_4$ over $\alpha_3$. Then, slide $\alpha_1$ and $\alpha_4$ over both $\alpha_2$ and $\alpha_3$. Finally, slide $\alpha_4$ over $\alpha_1$. For $\beta$ curves, slide $\beta_1$ over $\beta_3$. Then, slide $\beta_4$ over $\beta_1$. For $\gamma$ curves, slide $\gamma_4$ and $\gamma_1$ over $\gamma_2$. Then, slide $\gamma_4$ over $\gamma_1$. After these slides, the diagram is depicted in Figure \ref{fig:sec4_1}. In Figure \ref{fig:sec4_1}, $\alpha_i$ and $\gamma_i$ ($i=1,4$) are parallel to each other, and $\beta_4$ transversally intersects both $\alpha_4$ and $\gamma_4$ only once. Thus,  we can destabilize Figure \ref{fig:sec4_1}. From now, we deform Figure \ref{fig:sec4_1} to simplify $\beta_4$. %複数回やるところはそれを書いたほうがいいかも．

In Figure \ref{fig:sec4_1}, for $\alpha$ curves, slide both $\alpha_1$ and $\alpha_4$ over $\alpha_2$, after that, slide $\alpha_2$ over $\alpha_3$. For $\gamma$ curves, slide both $\gamma_1$ and $\gamma_4$ over $\gamma_2$. After these slides, Figure \ref{fig:sec4_2} is obtained. Note that in Figure \ref{fig:sec4_2}, $t_{\delta}(\beta_4)=m$ (the braid relation), $i(\alpha_{j}, \delta)=i(\gamma_{j}, \delta)=0$ ($j=1,2,4$) and $i(\beta_{k},\delta)=0$ ($k=1,2,3$), where $i(\cdot, \cdot)$ represents the geometric intersection number.

In Figure \ref{fig:sec4_2}, perform $t_{\delta}$. Then, slide $\alpha_3$ and $\gamma_3$ over $\alpha_4$ and $\gamma_4$, respectively so that $i(\alpha_3,\beta_4)=i(\gamma_3,\beta_4)=0$. Moreover, slide $\alpha_3$ and $\gamma_3$ over $\alpha_2$ and $\gamma_2$, respectively. After that, perform $t_{\beta_2}^{-1}$. Finally, for $\alpha$ curves, slide $\alpha_2$ over $\alpha_3$, and $\alpha_4$ over $\alpha_2$. For $\gamma$ curves, slide $\gamma_4$ over $\gamma_2$. After these slides and twists, Figure \ref{fig:sec4_3} is obtained.

In Figure \ref{fig:sec4_3}, since $\alpha_4$ and $\gamma_4$ are parallel to each other, and $\beta_4$ transversally intersects $\alpha_4$ and $\gamma_4$ once, by destabilizing them, Figure \ref{fig:sec4_4} is obtained.

In Figure \ref{fig:sec4_4}, by sliding each curve properly, we have Figure \ref{fig:sec4_5}, the stabilization of the genus 0 trisection diagram of $S^4$.

\begin{figure}[h]
\centering
\includegraphics[scale=0.35]{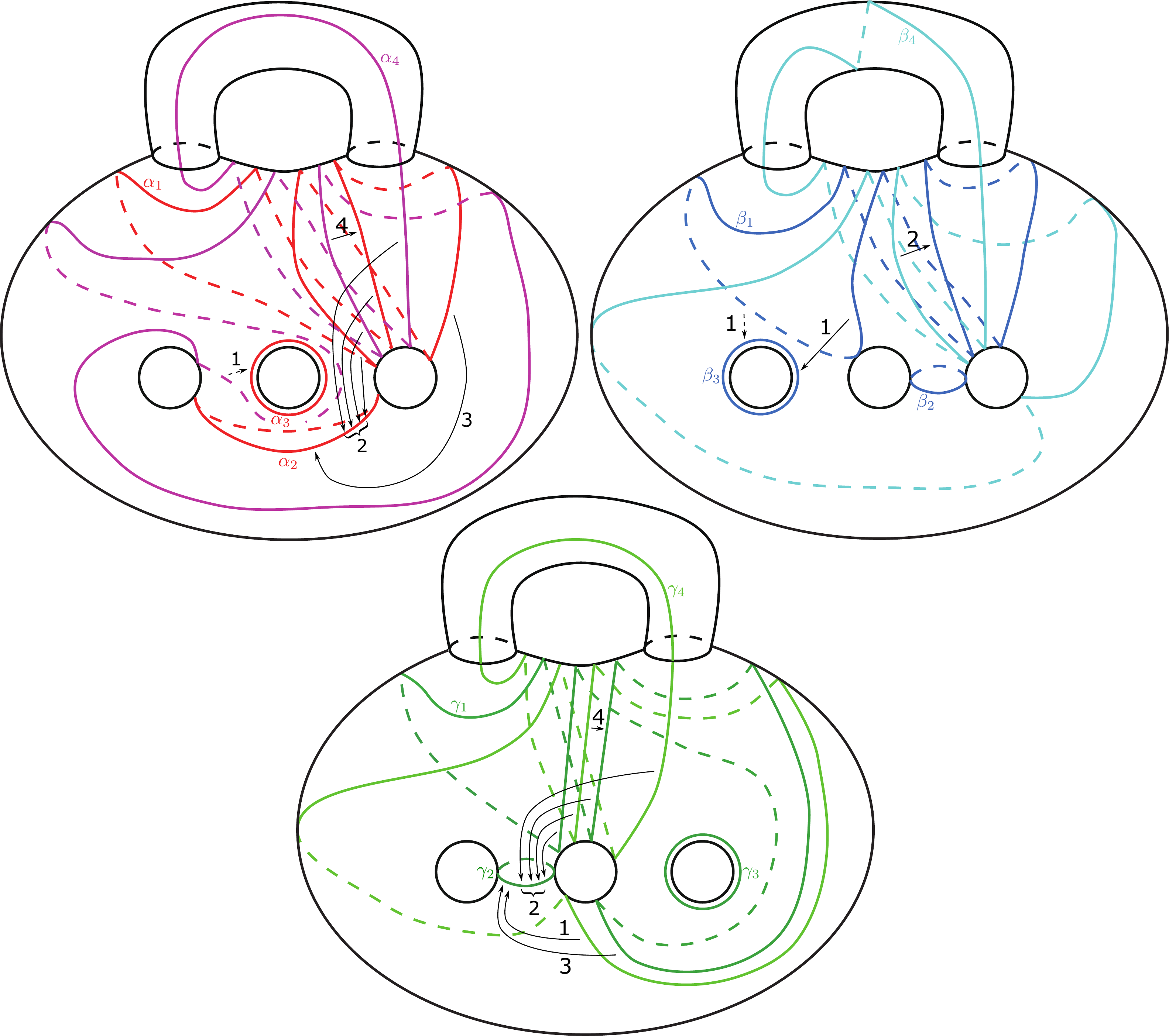}

\setlength{\captionmargin}{50pt}
\caption{The trisection diagram $\mathcal{D}_{b} \cup \mathcal{D}_{S(t(3,2))}$.}
\label{fig:start}
\end{figure}

\begin{figure}[h]
\centering
\includegraphics[scale=0.35]{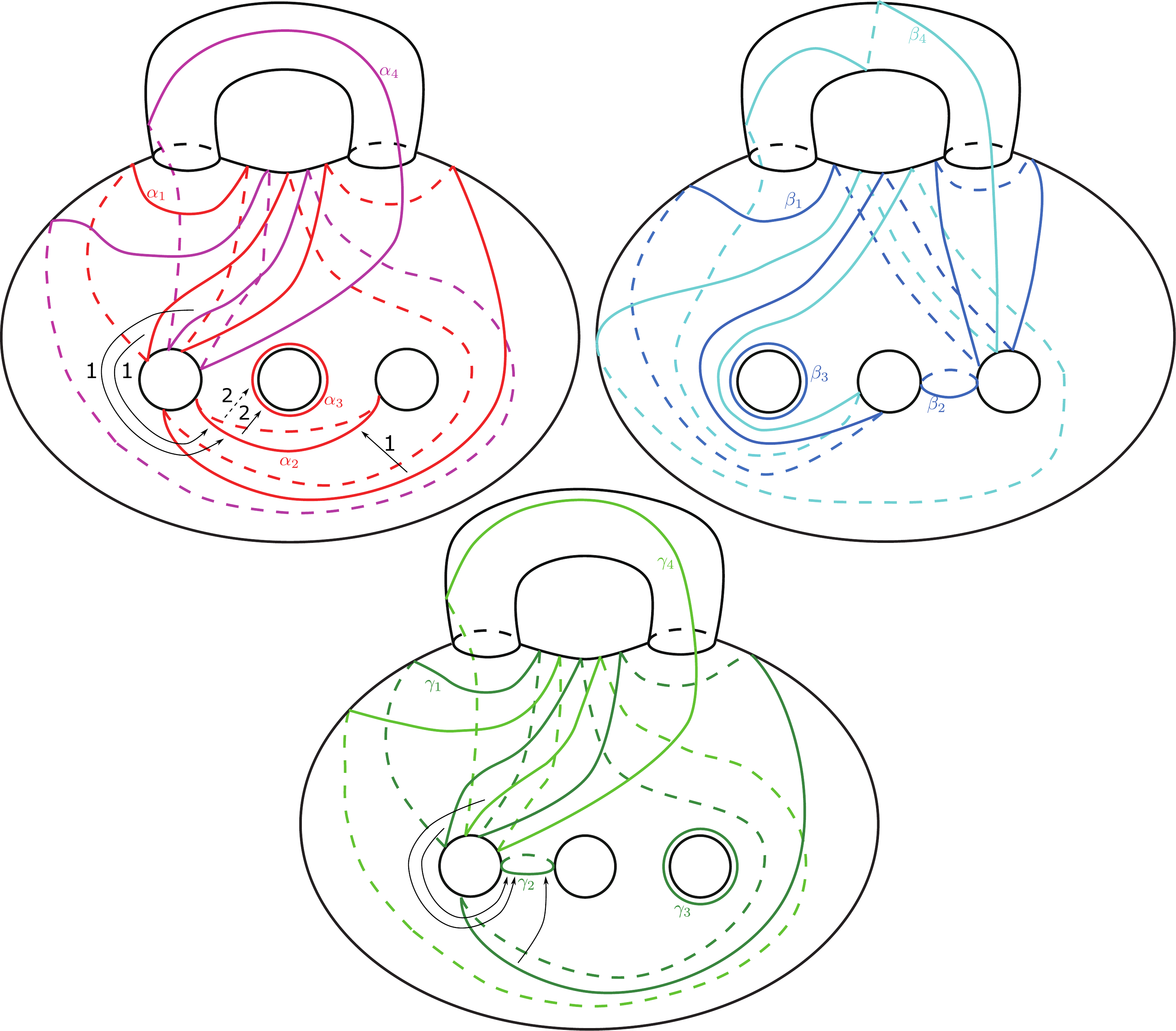}

\setlength{\captionmargin}{50pt}
\caption{A trisection diagram obtained from Figure \ref{fig:start} after some handle slides.}
\label{fig:sec4_1}
\end{figure}

\begin{figure}[h]
\centering
\includegraphics[scale=0.35]{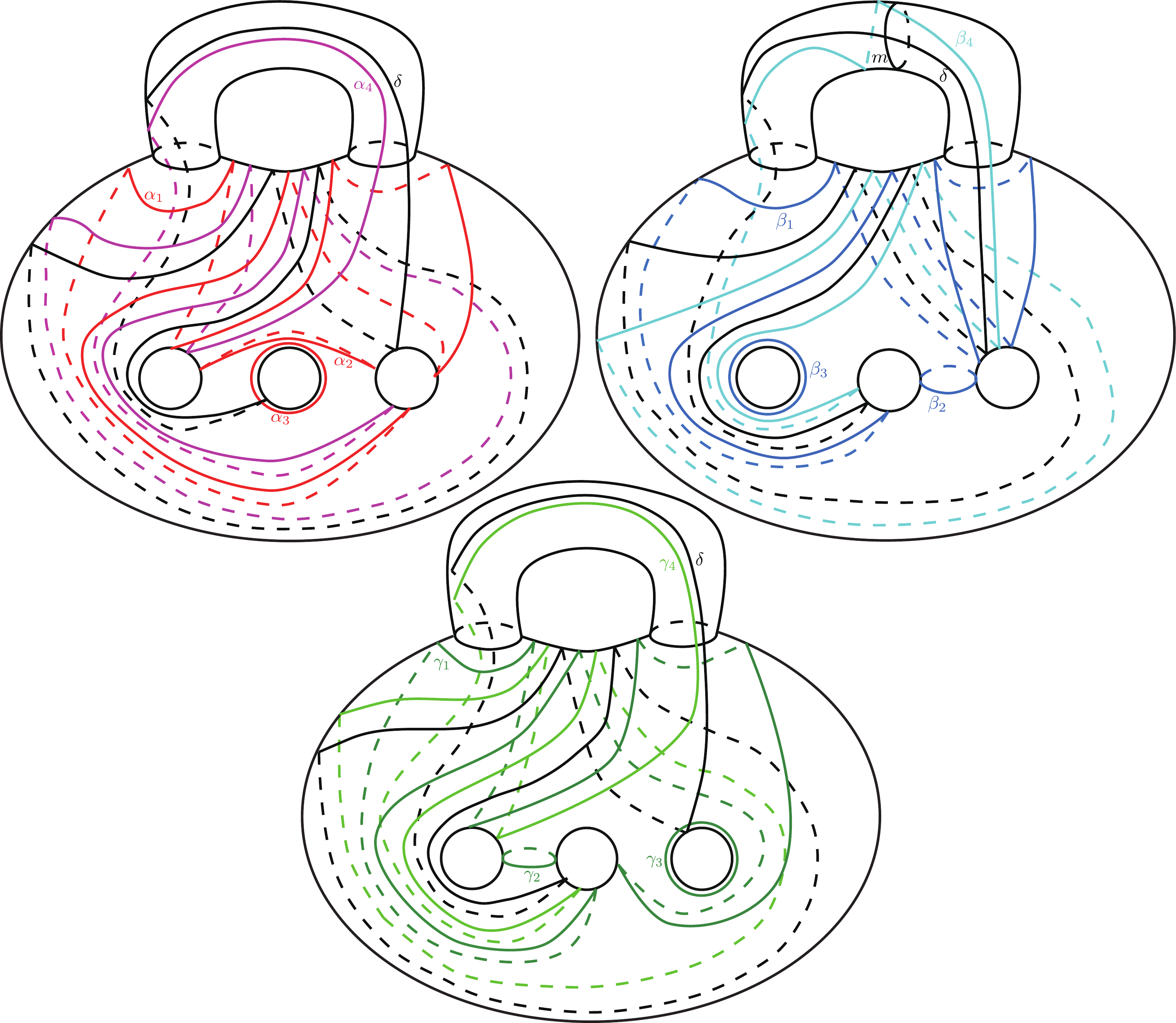}

\setlength{\captionmargin}{50pt}
\caption{A trisection diagram obtained from Figure \ref{fig:sec4_1} after some handle slides.}
\label{fig:sec4_2}
\end{figure}

\begin{figure}[h]
\centering
\includegraphics[scale=0.25]{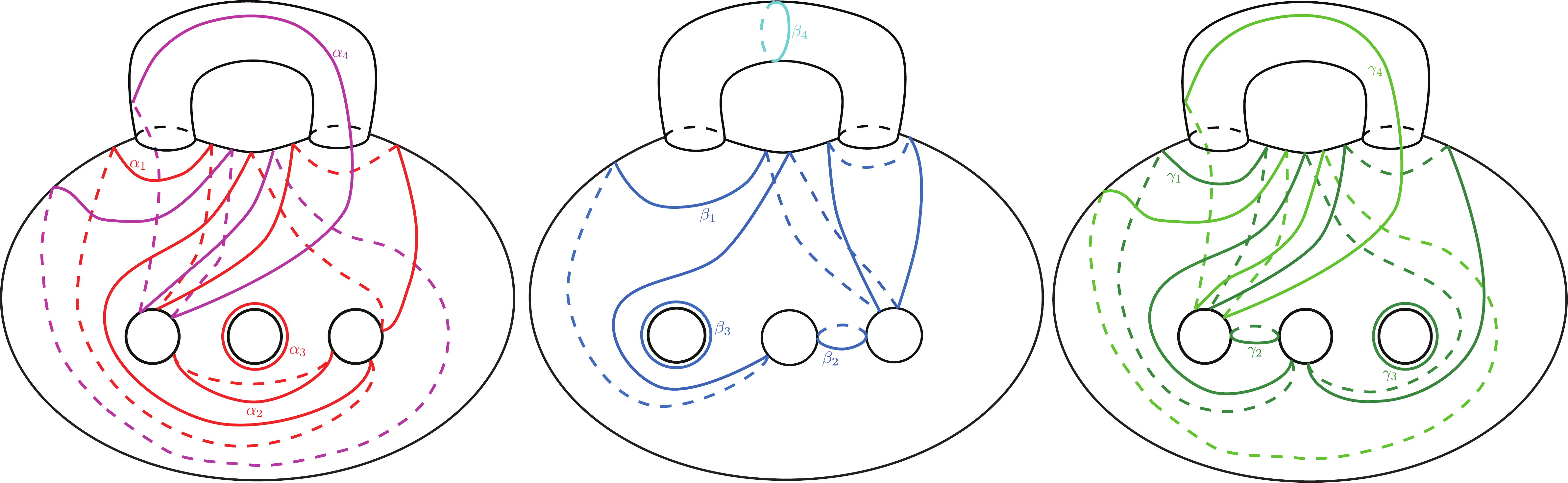}

\setlength{\captionmargin}{50pt}
\caption{Before the destabilization for $\alpha_4$, $\beta_4$ and $\gamma_4$.}
\label{fig:sec4_3}
\end{figure}

\begin{figure}[h]
\centering
\includegraphics[scale=0.25]{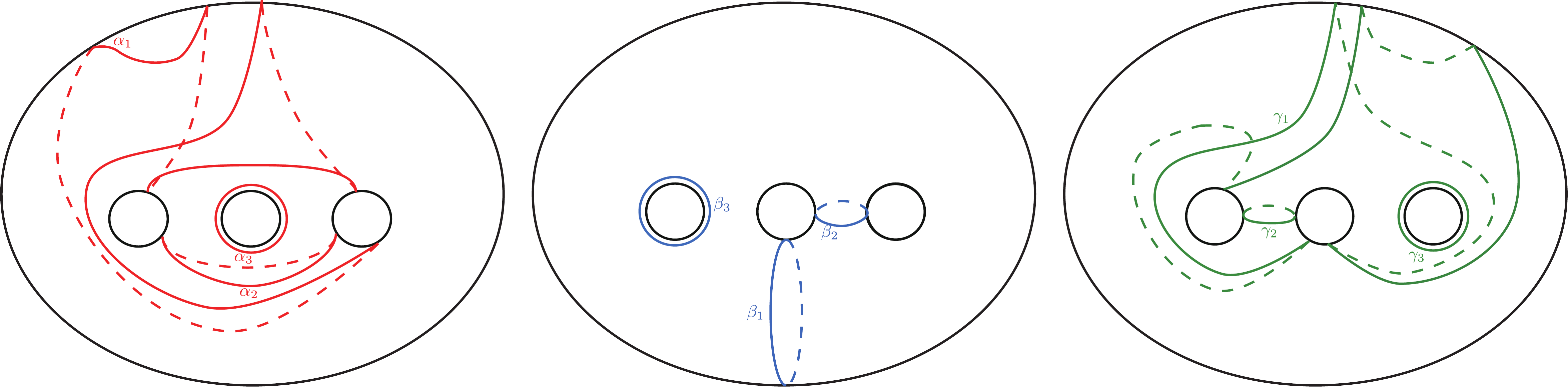}

\setlength{\captionmargin}{50pt}
\caption{After the destabilization.}
\label{fig:sec4_4}
\end{figure}

\begin{figure}[h]
\centering
\includegraphics[scale=0.25]{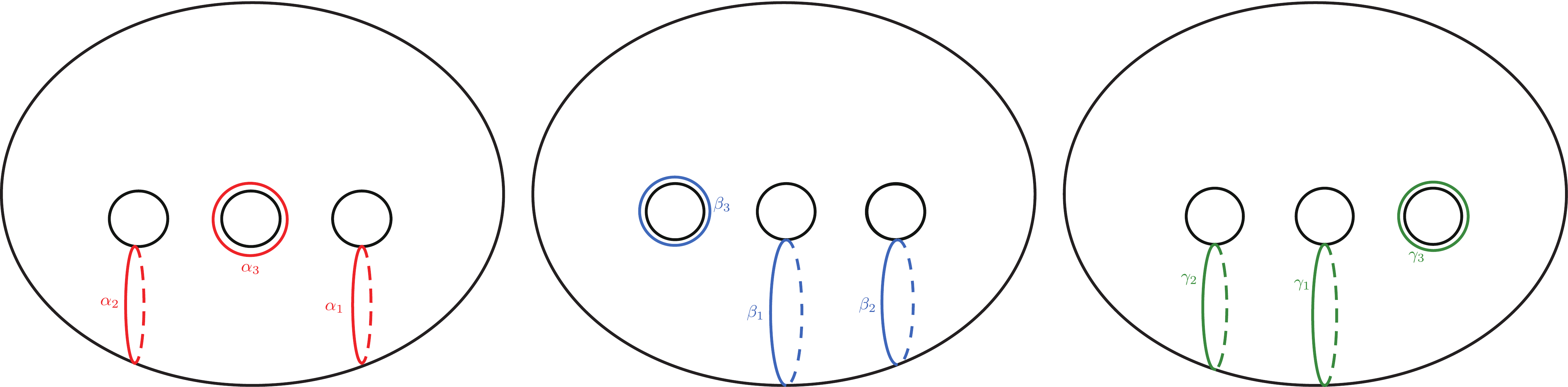}

\setlength{\captionmargin}{50pt}
\caption{The stabilization of the genus 0 trisection diagram of $S^4$.}
\label{fig:sec4_5}
\end{figure}

\begin{figure}[h]
\centering
\includegraphics[scale=0.35]{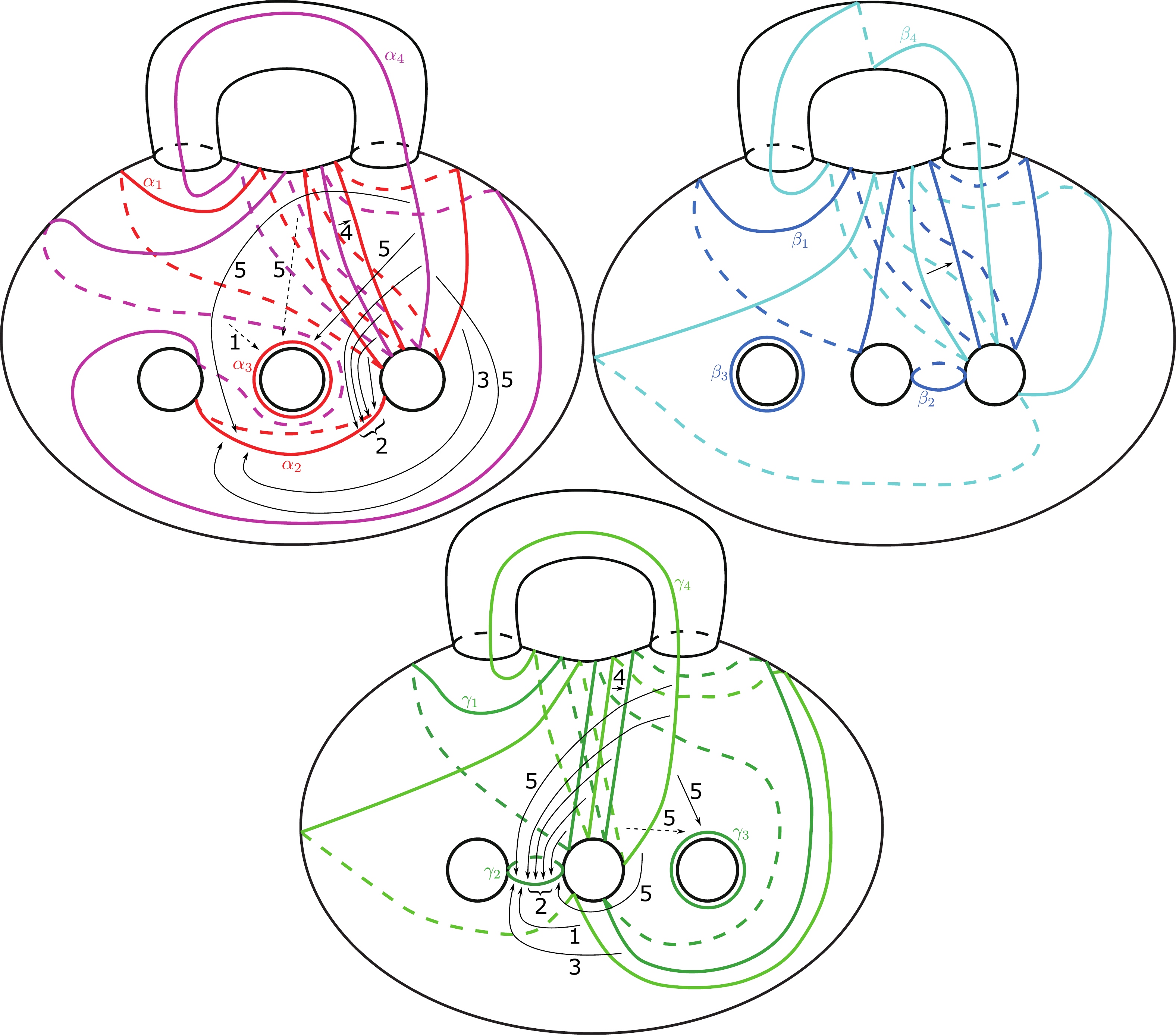}

\setlength{\captionmargin}{50pt}
\caption{The trisection diagram $\overline{\mathcal{D}_{b}} \cup \mathcal{D}_{S(t(3,2))}$.}
\label{fig:Gluck_overline{D_b}}
\end{figure}

\begin{figure}[h]
\centering
\includegraphics[scale=0.35]{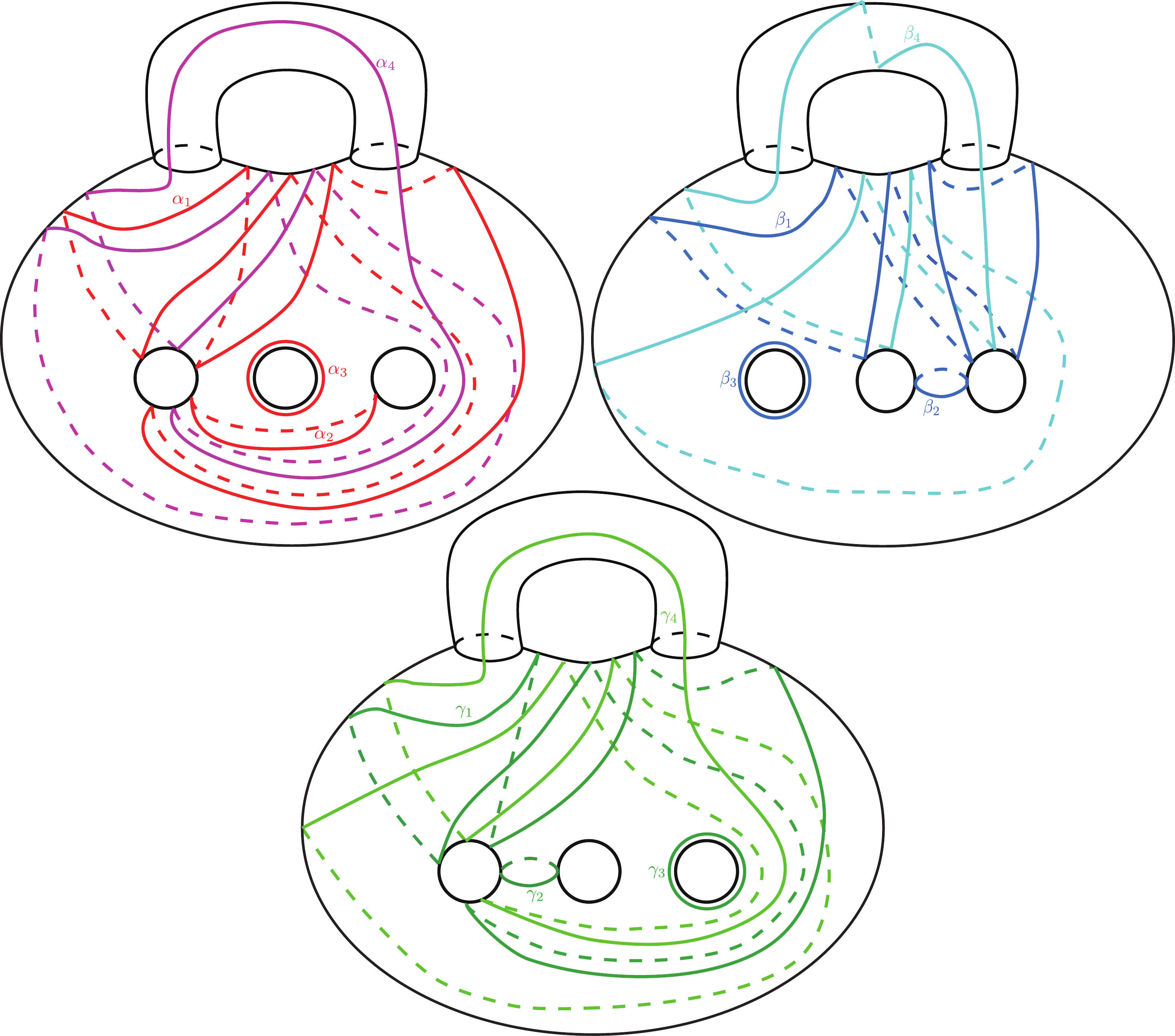}

\setlength{\captionmargin}{50pt}
\caption{A trisection diagram obtained from Figure \ref{fig:Gluck_overline{D_b}} by some handle slides.}
\label{fig:sec4_6}
\end{figure}

\begin{figure}[h]
\centering
\includegraphics[scale=0.35]{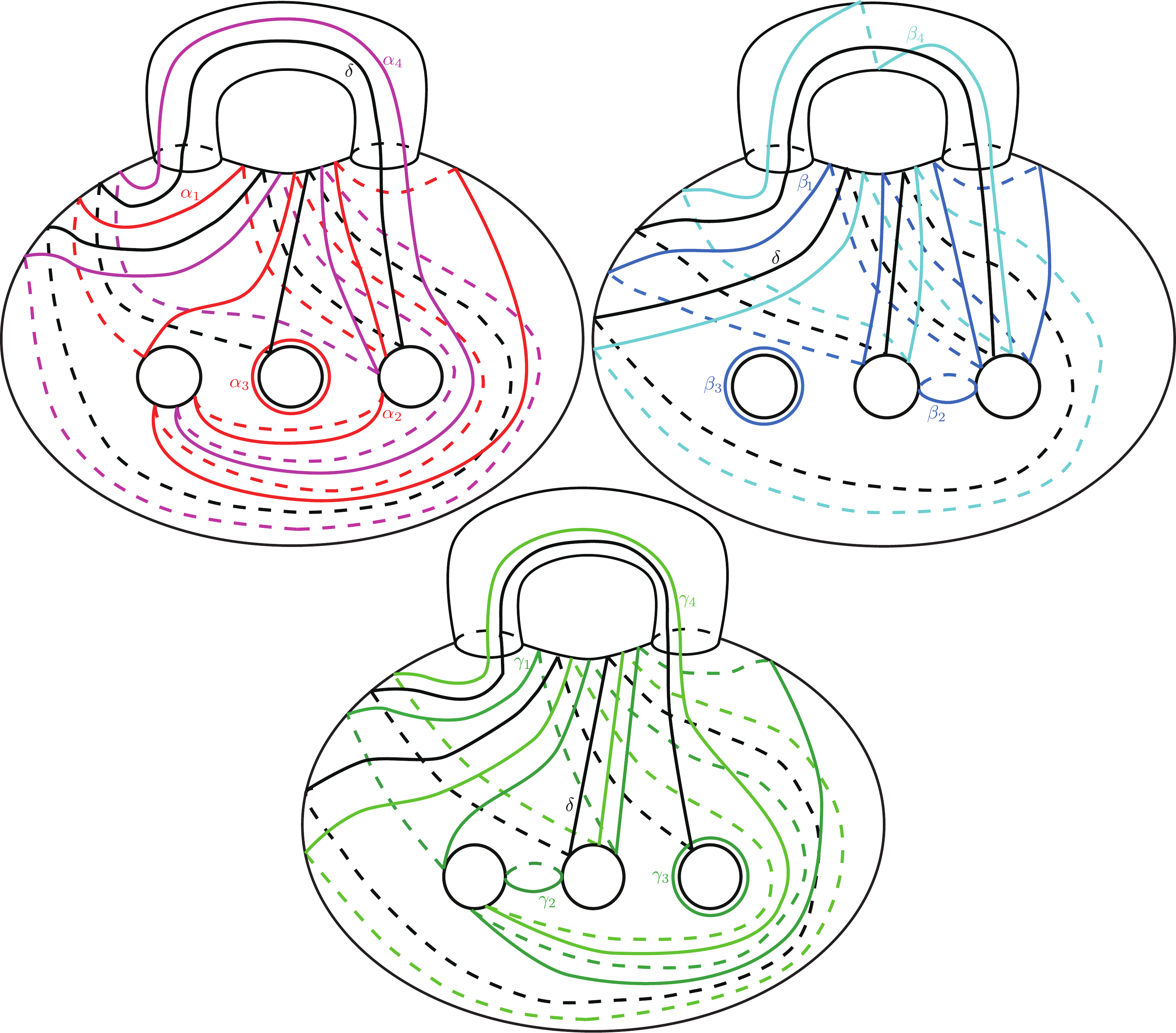}

\setlength{\captionmargin}{50pt}
\caption{A trisection diagram obtained from Figure \ref{fig:sec4_6} by some handle slides.}
\label{fig:sec4_7}
\end{figure}

\begin{figure}[h]
\centering
\includegraphics[scale=0.25]{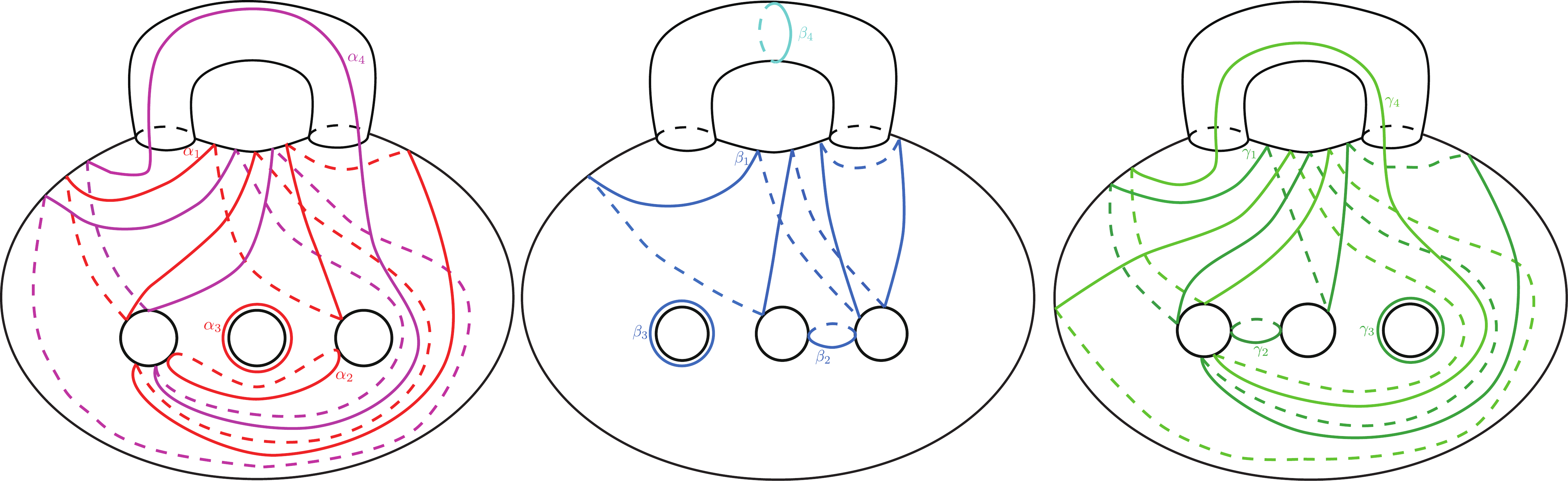}

\setlength{\captionmargin}{50pt}
\caption{Before the destabilization for $\alpha_4$, $\beta_4$ and $\gamma_4$.}
\label{fig:sec4_8}
\end{figure}

\begin{figure}[h]
\centering
\includegraphics[scale=0.25]{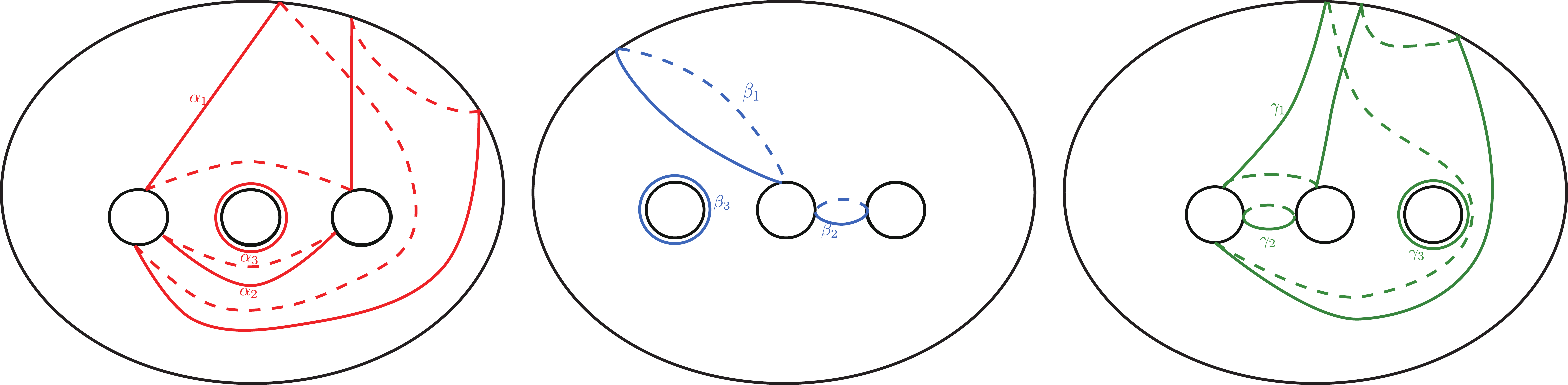}

\setlength{\captionmargin}{50pt}
\caption{After the destabilization.}
\label{fig:sec4_9}
\end{figure}

Next, we show the case of $\overline{\mathcal{D}_{b}}$. In Figure \ref{fig:Gluck_overline{D_b}}, for $\alpha$ curves, slide $\alpha_4$ over $\alpha_3$. Then, slide $\alpha_1$ and $\alpha_4$ over both $\alpha_2$ and $\alpha_3$. After that, slide $\alpha_1$ over $\alpha_2$. %ここまでD_bと同じ．
Then, slide $\alpha_4$ over $\alpha_1$. After that, slide $\alpha_4$ over both $\alpha_2$ and $\alpha_3$. %赤はこれでOK. 
For $\beta$ curves, slide $\beta_4$ over $\beta_1$. For $\gamma$ curves, slide $\gamma_4$ over $\gamma_2$. Then, slide $\gamma_1$ and $\gamma_4$ over $\gamma_2$. After that, slide $\gamma_4$ over $\gamma_1$. Then, slide $\gamma_4$ over both $\gamma_2$ and $\gamma_3$.  After these slides, the diagram is depicted in Figure \ref{fig:sec4_6}. In Figure \ref{fig:sec4_6}, $\alpha_i$ and $\gamma_i$ ($i=1,4$) are parallel to each other, and $\beta_4$ transversally intersects $\alpha_4$ and $\gamma_4$ only once. Thus, we can destabilize Figure \ref{fig:sec4_6}. From now, we deform Figure \ref{fig:sec4_6} to simplify $\beta_4$.

In Figure \ref{fig:sec4_6}, slide $\alpha_1$ and $\alpha_4$ over both $\alpha_2$ and $\alpha_3$, and $\gamma_1$ and $\gamma_4$ over $\gamma_2$. Then, Figure \ref{fig:sec4_7} is obtained. Note that in Figure \ref{fig:sec4_7}, $t_{\delta}^{-1}(\beta_4)=m$ (the braid relation), $i(\alpha_{j}, \delta)=i(\gamma_{j}, \delta)=0$ ($j=1,2,4$) and $i(\beta_{k},\delta)=0$ ($k=1,2,3$), where $m$ is the curve depicted in Figure \ref{fig:sec4_2}.

In Figure \ref{fig:sec4_7}, perform $t_{\delta}^{-1}$. Then, slide $\alpha_3$ and $\gamma_3$ over $\alpha_4$ and $\gamma_4$, respectively so that $i(\alpha_3,\beta_4)=i(\gamma_3,\beta_4)=0$. Moreover, slide $\alpha_3$ and $\gamma_3$ over $\alpha_2$ and $\gamma_2$, respectively. After that, perform $t_{\beta_2}$. Finally, for $\alpha$ curves, slide $\alpha_4$ over both $\alpha_2$ and $\alpha_3$. For $\gamma$ curves, slide $\gamma_4$ over $\gamma_2$. After these slides and twists, Figure \ref{fig:sec4_8} is obtained.

In Figure \ref{fig:sec4_8}, since $\alpha_4$ and $\gamma_4$ are parallel to each other, and $\beta_4$ transversally intersects $\alpha_4$ and $\gamma_4$ only  once, by destabilizing them, Figure \ref{fig:sec4_9} is obtained.

In Figure \ref{fig:sec4_9}, by sliding each curve properly, we have Figure \ref{fig:sec4_5}, the stabilization of the genus 0 trisection diagram of $S^4$.

\end{proof}

\section{Homological calculations}
In this section, we compute homological information of trisection diagrams of $S^4$ obtained by Gluck twisting along the spun $(2n+1,2)$-torus knots. 
This section is organized as follows: First, we investigate how the homology class of curves in a given trisection diagram changes with handle slides and Dehn twists.
After that, we define a homologically standard trisection diagram of $S^4$ and show that all trisection diagrams in section 3 are homologically standard trisection diagrams of $S^4$. 

The following lemma is a change of a homology class of curves on a surface after handle sliding.
\begin{lemma}
	Let $C_1$ and $C_2$ be non-isotopic essential simple closed curves on an orientable closed surface $\Sigma_g$. %and $\{\alpha_1, \beta_1, . . . , \alpha_g, \beta_g\}$ a symplectic basis of $H_1(\Sigma_g)$ such that $\alpha_i$ and $\beta_i$ are dual for $i\in\{1,\ldots , g\}$.
	Suppose that $C_3$ is an essential simple closed curve that is obtained by handle sliding $C_1$ over $C_2$.
	Then, 
	\[
			[C_3]=[C_1]\pm[C_2]
	\]
	in $H_1(\Sigma_g)$.
\end{lemma}
\begin{proof}
	Let $a$ be an arc in $\Sigma_g$ which connects $C_1$ and $C_2$ and $\{l_1, \ldots ,l_{2g}\}$ a symplectic basis of $H_1(\Sigma_g)$.
	Then $C_3$ is one of the boundaries of $N(C_1\cup a\cup C_2)$ that is not isotopic to both $C_1$ and $C_2$.
	The intersection of $a$ and $l_i$ corresponds to an intersection of $C_3$ and $l_i$.
	One of the signs of two points corresponding to this intersection is positive, the other is negative. 
	Hence the algebraic intersection number of $C_3$ and $l_i$ is $\langle C_1, l_i\rangle+\langle C_2, l_i \rangle$, where $\langle \cdot, \cdot \rangle$ is a algebraic intersection number.
	Therefore the statement holds.
\end{proof}

\begin{definition}
	Let $\Sigma_g$ be an orientable closed surface and $[C_1]$ and $[C_2]$ non-trivial elements in  $H_1(\Sigma_g)$, and $[C_3]=[C_1]\pm[C_2]$.
	We call the operation that replaces $[C_1]$ by $[C_3]$ {\it a homological handle slide} $[C_1]$ over $[C_2]$.
\end{definition}

\begin{lemma}
	Let $C_1$ and $C_2$ be essential simple closed curves on an orientable closed surface $\Sigma_g$. 
	Suppose that $C_3$ is an essential simple closed curve that is obtained from $C_1$ by Dehn twisting along $C_2$.
	Then, 
	\[
			[C_3]=[C_1]\pm\langle C_1, C_2 \rangle [C_2]
	\]
	in $H_1(\Sigma_g)$, where  $\langle C_1, C_2 \rangle$ is the algebraic intersection number of $C_1$ and $C_2$ for certain basis of $H_1(\Sigma_g)$.
\end{lemma}
\begin{proof}
	If $C_1$ intersects $C_2$ exactly one point, $[C_3]=[C_1]\pm [C_2]$ after the Dehn twists (the sign depends on both Dehn twist is the right or left hand, and $\langle C_1, C_2\rangle$ is $+1$ or $-1$). If $\langle C_1, C_2 \rangle$ is not $\pm 1$, we shall consider Dehn twist along $C_2$ for each intersection $C_1\cap C_2$. The homology class of obtained curve is change $\pm 1$ for the intersection with positive intersection and the other is $\mp 1$. Hence homology class of obtained curve is $[C_1]\pm\langle C_1, C_2 \rangle [C_2]$.
\end{proof}

\begin{definition}
	Let $\Sigma_g$ be an orientable closed surface and $[C_1]$ and $[C_2]$ non-trivial elements in  $H_1(\Sigma_g)$, and $[C_3]=[C_1]\pm\langle C_1, C_2 \rangle [C_2]$.
	We say the operation that replaces $[C_1]$ by $[C_3]$ {\it a homological Dehn twist} along $[C_2]$.
\end{definition}

\begin{definition}
	Let $(\Sigma; \alpha, \beta, \gamma)$ be a trisection diagram of the 4-sphere and $\alpha=\{\alpha_1,\ldots, \alpha_g\}$, $\beta=\{\beta_1, \ldots , \beta_g\}$ and $\gamma=\{\gamma_1, \ldots , \gamma_g\}$.
	We say $(\Sigma; \alpha, \beta, \gamma)$ is {\it homologically standard} if the following holds after performing homological Dehn twists and handle slides finitely many times, and replacing some indices:
	
	\begin{itemize}
		\item $\langle [\alpha_i], [\beta_i] \rangle=\langle [\alpha_i], [\gamma_i] \rangle=\pm 1$, $\langle [\alpha_i], [\beta_j] \rangle=\langle [\gamma_i], [\beta_j] \rangle=0$ and $[\beta_i]=[\gamma_i]$ in $H_1(\Sigma)$ for $i\neq j$.
		\item $\langle [\beta_i], [\gamma_i] \rangle=\langle [\beta_i], [\alpha_i] \rangle=\pm 1$, $\langle [\beta_i], [\gamma_j] \rangle=\langle [\beta_i], [\alpha_j] \rangle=0$ and $[\gamma_i]=[\alpha_i]$ in $H_1(\Sigma)$ for $i\neq j$.
		\item $\langle [\gamma_i], [\alpha_i] \rangle=\langle [\gamma_i], [\beta_i] \rangle=\pm 1$, $\langle [\gamma_i], [\alpha_j] \rangle=\langle [\gamma_i], [\beta_j] \rangle=0$ and $[\alpha_i]=[\beta_i]$ in $H_1(\Sigma)$ for $i\neq j$.
	\end{itemize}
\end{definition}
Note that if a trisection diagram is standard, then it is also homologically standard, and this property serves as a kind of invariant for measuring standardness.
\begin{proof}[Proof of Theorem \ref{thm2}]
	Let $\{C_1,  \ldots, C_{2g}\}$ be a basis of $\Sigma_g$.
	We define the matrix $T$ whose row vectors are $\alpha_i$, $\beta_i$ and $\gamma_i$ in $H_1(\Sigma_g)$ with respect to a basis $\{C_1,  \ldots, C_{2g}\}$ for $i=1, 2, 3, 4$.
	\[
			T :=\begin{pmatrix}
					\alpha_1\\
					\alpha_2\\
					\alpha_3\\
					\alpha_4\\
					\beta_1\\
					\beta_2\\
					\beta_3\\
					\beta_4\\
					\gamma_1\\
					\gamma_2\\
					\gamma_3\\
					\gamma_4
				\end{pmatrix}
	\]
	
	By Figure \ref{fig:start}, and Lemma \ref{lem:main}, we obtain
	\[
			T=\begin{pmatrix}
					0&0&-1&-1&0&0&0&0\\
					1&0&-1&0&0&0&0&0\\
					0&0&0&0&0&1&0&0\\
					1&0&-2(1+n)&3+2n&1&1&1&-1\\
					\hline
					0&-1&0&-1&0&0&0&0\\
					0&1&-1&0&0&0&0&0\\
					0&0&0&0&1&0&0&0\\
					0&0&-(1+2n)&2(2+n)&1&1&1&-1\\
					\hline
					0&-1&0&-1&0&0&0&0\\
					1&-1&0&0&0&0&0&0\\
					0&0&0&0&0&0&1&0\\
					0&-(1+2n)&0&3+2n&1&1&1&-1
				\end{pmatrix}.
	\]
	From now, we will perform homological handle slides and Dehn twists until $\alpha_i$, $\beta_i$, and $\gamma_i$ are homologically standard for $i=1, 2, 3, 4$.
	After Dehn twisting $1+2n$ times along $(0, 0, 1, 0, 0, 0, 0, 0)$, we have
	\[
			T=\begin{pmatrix}
					0&0&-1&-1&0&0&0&0\\
					1&0&-1&0&0&0&0&0\\
					0&0&0&0&0&1&0&0\\
					1&0&-1&3+2n&1&1&1&-1\\
					\hline
					0&-1&0&-1&0&0&0&0\\
					0&1&-1&0&0&0&0&0\\
					0&0&0&0&1&0&0&0\\
					0&0&0&2(2+n)&1&1&1&-1\\
					\hline
					0&-1&0&-1&0&0&0&0\\
					1&-1&0&0&0&0&0&0\\
					0&0&1+2n&0&0&0&1&0\\
					0&-(1+2n)&1+2n&3+2n&1&1&1&-1
				\end{pmatrix}.
	\]
	Next, we perform homological Dehn twists $3+2n$ times along $(0, 0, 0, 1, 0, 0, 0, 0)$. Then,  we obtain
	\[
			T=\begin{pmatrix}
					0&0&-1&-1&0&0&0&0\\
					1&0&-1&0&0&0&0&0\\
					0&0&0&0&0&1&0&0\\
					1&0&-1&0&1&1&1&-1\\
					\hline
					0&-1&0&-1&0&0&0&0\\
					0&1&-1&0&0&0&0&0\\
					0&0&0&0&1&0&0&0\\
					0&0&0&1&1&1&1&-1\\
					\hline
					0&-1&0&-1&0&0&0&0\\
					1&-1&0&0&0&0&0&0\\
					0&0&1+2n&0&0&0&1&0\\
					0&-(1+2n)&1+2n&0&1&1&1&-1
				\end{pmatrix}.
	\]
	After homological handle sliding $\alpha_4$ over $\alpha_2$, $\beta_2$ over $\beta_1$ and $\gamma_4$ over $\gamma_2$, we obtain
	\[
			T=\begin{pmatrix}
					0&0&-1&-1&0&0&0&0\\
					1&0&-1&0&0&0&0&0\\
					0&0&0&0&0&1&0&0\\
					-(1+2n)&0&1+2n&0&1&1&1&-1\\
					\hline
					0&-1&0&-1&0&0&0&0\\
					0&0&-1&-1&0&0&0&0\\
					0&0&0&0&1&0&0&0\\
					0&0&0&1&1&1&1&-1\\
					\hline
					0&-1&0&-1&0&0&0&0\\
					1&-1&0&0&0&0&0&0\\
					0&0&1+2n&0&0&0&1&0\\
					-(1+2n)&0&1+2n&0&1&1&1&-1
				\end{pmatrix}.
	\]
	Next, we perform homological handle slides $\alpha_2$ over $\alpha_1$ and $\gamma_2$ over $\gamma_1$. Then, we obtain
	\[
			T=\begin{pmatrix}
					0&0&-1&-1&0&0&0&0\\
					1&0&0&1&0&0&0&0\\
					0&0&0&0&0&1&0&0\\
					-(1+2n)&0&1+2n&0&1&1&1&-1\\
					\hline
					0&-1&0&-1&0&0&0&0\\
					0&0&-1&-1&0&0&0&0\\
					0&0&0&0&1&0&0&0\\
					0&0&0&1&1&1&1&-1\\
					\hline
					0&-1&0&-1&0&0&0&0\\
					1&0&0&1&0&0&0&0\\
					0&0&1+2n&0&0&0&1&0\\
					-(1+2n)&0&1+2n&0&1&1&1&-1
				\end{pmatrix}.
	\]
	After homological handle sliding $\alpha_4$ over $\alpha_2$ and $\gamma_4$ over $\gamma_2$, we obtain
		\[
			T=\begin{pmatrix}
					0&0&-1&-1&0&0&0&0\\
					1&0&0&1&0&0&0&0\\
					0&0&0&0&0&1&0&0\\
					0&0&1+2n&1+2n&1&1&1&-1\\
					\hline
					0&-1&0&-1&0&0&0&0\\
					0&0&-1&-1&0&0&0&0\\
					0&0&0&0&1&0&0&0\\
					0&0&0&1&1&1&1&-1\\
					\hline
					0&-1&0&-1&0&0&0&0\\
					1&0&0&1&0&0&0&0\\
					0&0&1+2n&0&0&0&1&0\\
					0&0&1+2n&1+2n&1&1&1&-1
				\end{pmatrix}.
	\]
	Next, we perform homological Dehn twists $-(2n+1)$ times along $(0, 0, 1, 0, 0, 0, 0, 0)$. Then, we obtain
	\[
			T=\begin{pmatrix}
					0&0&-1&-1&0&0&0&0\\
					1&0&0&1&0&0&0&0\\
					0&0&0&0&0&1&0&0\\
					0&0&0&1+2n&1&1&1&-1\\
					\hline
					0&-1&0&-1&0&0&0&0\\
					0&0&-1&-1&0&0&0&0\\
					0&0&0&0&1&0&0&0\\
					0&0&-(2n+1)&1&1&1&1&-1\\
					\hline
					0&-1&0&-1&0&0&0&0\\
					1&0&0&1&0&0&0&0\\
					0&0&0&0&0&0&1&0\\
					0&0&0&1+2n&1&1&1&-1
				\end{pmatrix}.
	\]
	Finally, we perform homological Dehn twists $-(1+2n)$ times along $(0, 0, 0, 1, 0, 0, 0, 0)$, and after homological handlesliding $\beta_4$ over $\beta_2$, we obtain
	\[
			T=\begin{pmatrix}
					0&0&-1&-1&0&0&0&0\\
					1&0&0&1&0&0&0&0\\
					0&0&0&0&0&1&0&0\\
					0&0&0&0&1&1&1&-1\\
					\hline
					0&-1&0&-1&0&0&0&0\\
					0&0&-1&-1&0&0&0&0\\
					0&0&0&0&1&0&0&0\\
					0&0&0&1&1&1&1&-1\\
					\hline
					0&-1&0&-1&0&0&0&0\\
					1&0&0&1&0&0&0&0\\
					0&0&0&0&0&0&1&0\\
					0&0&0&0&1&1&1&-1
				\end{pmatrix}.
	\]
	After reordering the indices, $\alpha_i$, $\beta_i$, and $\gamma_i$ are homologically standard for $i=1, 2, 3, 4$.
	
	We can prove similarly that in the case of $\overline{\mathcal{D}_b}$.
\end{proof}

\section{Acknowledgements}

The first author would like to express sincere gratitude to his supervisor, Hisaaki Endo, for his unwavering support and encouragement throughout the project. The first author was supported by JST, the establishment of university fellowships towards the creation of science technology innovation, Grant Number JPMJFS2112, and is partially supported by Grant-in-Aid for JSPS Research Fellow from JSPS KAKENHI Grant Number JP23KJ0888.
The second author was partially supported by Grant-in-Aid for JSPS Research Fellow from JSPS KAKENHI Grant Number JP20J20545.

\end{document}